\documentclass{amsart}[12pt]
\usepackage{amssymb,amsmath,amscd,graphicx, color,picture}
\usepackage[hidelinks]{hyperref}
\usepackage{blkarray, tikz}

\makeindex 

\newtheorem{theorem}{Theorem}[section]
\newtheorem{lemma}[theorem]{Lemma}
\newtheorem{corollary}[theorem]{Corollary}
\newtheorem{definition}[theorem]{Definition}

\newtheorem{remark}[theorem]{\it Remark}
\newtheorem{example}[theorem]{Example}
\newtheorem{proposition}[theorem]{Proposition}

\numberwithin{equation}{subsection}

\def\tree{\sigma}
\def\gr{\mathrm{gr}}
\def\Gr{\mathrm{Gr}}

\def\SL{\mathrm{SL}}

\def\trop{\mathrm{Trop}}
\def\G{\mathbb{G}}

\def\C{\mathbb{C}}
\def\R{\mathbb{R}}
\def\Z{\mathbb{Z}}

\def\P{\mathbb{P}}
\def\Q{\mathbb{Q}}
\def\A{\mathbb{A}}
\def\B{\mathbb{B}}
\def\O{\mathcal{O}}
\def\L{\mathcal{L}}

\def\r{\mathbf{r}}
\def\d{\mathbf{d}}

\def\a{\mathbf{a}}

\def\bx{\mathbf{x}}

\def\s{\mathbf{s}}
\def\v{\mathfrak{v}}
\def\w{\mathfrak{w}}
\def\o{\mathbf{o}}

\def\ql{\backslash \! \backslash}
\def\qr{\slash \! \slash}

\newif\ifdebug                                                      %
\debugtrue

\begin{document}

\title[Grassmannian compactification]{Tropical geometry and Newton-Okounkov cones for Grassmannian of Planes from compactifications}
\author{Christopher Manon}
\address{University of Kentucky}
\email{Christopher.Manon@uky.edu}

\author{Jihyeon Jessie Yang}
\address{Marian University - Indianapolis}
\email{jyang@marian.edu}

\begin{abstract}
We construct a family of compactifications of the affine cone of the Grassmannian variety of $2$-planes.  We show that both the tropical variety of the Pl\"ucker ideal and familiar valuations associated to the construction of Newton-Okounkov bodies for the Grassmannian variety can be recovered from these compactifications.  In this way, we unite various perspectives for constructing toric degenerations of flag varieties.  
\end{abstract}

\maketitle

\tableofcontents

\section{Introduction}

The study of toric degenerations of Flag varieties is a meeting point for techniques from commutative algebra, algebraic geometry, and representation theory.  Grassmannian varieties, in particular, being that they are often the most straightforward case to study after projective space, provide a testing ground for new constructions of toric degenerations as well as a tractable class of examples for comparisons.   A survey of recent activity in this area can be found in \cite{Fang-Fourier-Littelmann}.    

 In this paper, we study the Grassmannian variety $Gr_2(\C^n)$ of $2$-planes in $\C^n$.   Let $I_{2, n}$ be the Pl\"ucker ideal which cuts out the affine cone $X \subset \A^{\binom{n}{2}}$ of $Gr_2(\C^n)$.     Speyer and Sturmfels (\cite{Speyer-Sturmfels}) provide a comprehensive understanding of the known toric degenerations of $Gr_2(\C^n)$, which are constructed from initial ideals of $I_{2, n}$ and organized by tropical geometry.  In particular, in \cite{Speyer-Sturmfels} it is shown that the maximal cones of the tropical variety $\trop(I_{2, n})$ are in bijection with trivalent trees $\sigma$ with $n$ ordered leaves labeled with $1, \ldots, n$, and that the initial ideal associated with each of these cones is prime and binomial.    We present a distinct construction of this class of well-known toric degenerations using the representation theory of $\SL_2$ (Section \ref{toric}), a quiver variety-style construction of $X$, and a family of compactifications $X_\sigma$ (Section \ref{GIT}), one for each trivalent tree $\sigma$.  Our main result is the following. See \cite{Kaveh-Manon-NOK} and Section \ref{valuations} for the notion of Khovanskii basis. 

\begin{theorem}\label{main}
For each trivalent tree $\sigma$ as above, and a total ordering $<$ on the edges of $\sigma$, we construct:\\
\begin{enumerate}
\item a simplicial cone $C_\sigma$ of discrete, rank $1$ valuations on $\C[X]$ with common Khovanskii basis given by the Pl\"ucker generators of $\C[X]$,\\
\item a rank $2n-3$ discrete valuation $\v_{\sigma, <}$ on $\C[X]$ with Khovanskii basis given by the Pl\"ucker generators of $\C[X]$,\\
\item a compactification $X \subset X_\sigma$ by a combinatorial normal crossings divisor $D_\sigma$ such that $C_\sigma$ is spanned by the divisorial valuations associated to the components of $D_\sigma$, and $\v_{\sigma, <}$ is a Parshin point  valuation (see Section \ref{ppoint}) built from a flag of subvarieties of $X_\sigma$ obtained by intersection components of $D_\sigma$.\\ 
\end{enumerate}

\noindent
Furthermore, the affine semigroup algebra $\C[S_\sigma]$ associated to the value semigroup $S_\sigma$ of $\v_{\sigma, <}$ is presented by the initial ideal corresponding to the cone in $\trop(I_{2, n})$ associated to $\sigma$. 
\end{theorem}

\begin{remark}
We observe that by \cite[Lemma 3]{Kaveh-Manon-NOK}, the valuation $\v_{\sigma, <}$ coincides with any homogeneous valuation with value semigroup $S_\sigma$ constructed by one of the many methods used for constructing degenerations of flag varieties. 
\end{remark}

The compactification $X_\sigma$ has a natural description in terms of the geometry of $X$.  In Section \ref{GIT} we construct $X$ as a type of quiver variety coming from a choice of directed structure on the tree $\sigma$.  In particular, each edge of $\sigma$ is assigned a space, either $\SL_2$ or $\A^2$.  The compactification $X_\sigma$ is then the space where these edge coordinates are allowed to take values in a compactification of $\SL_2$ or $\A^2$. We also show that $X_\sigma$ is always Fano (Proposition \ref{Fano}).

To explain our results we recall the elements of two general theories underlying toric degeneration constructions.  As a set, the Berkovich analytification $X^{an}$ of an affine variety (\cite{Berkovich}) is the collection of all rank $1$ valuations on the coordinate ring $\C[X]$ which restrict to the trivial valuation on $\C$.  If $\mathcal{F} = \{f_1, \ldots, f_n\}$ is a set of generators of $\C[X]$, it is well-known (see \cite{Payne}) that the evaluation map $ev_{\mathcal{F}}$ which sends $v \in X^{an}$ to $(v(f_1), \ldots, v(f_n))$ maps onto the tropical variety $\trop(I)$ of $I$,  the ideal of forms which vanish on $\mathcal{F}$.  It is difficult to find a section of this map. The main result of the work of Cueto, Habich, and Werner \cite{CHW} carries out such a construction in the projective setting for $Gr_2(\C^n)$.  Part $(1)$ of Theorem \ref{main} extends to a version of this result on the affine cone $X$.  We define a polyhedral complex of trees $\mathcal{T}(n)$ in Section \ref{coneofvaluations} which is close (up to a lineality space) to the Biller-Holmes-Vogtmann space of phylogenetic trees \cite{Speyer-Sturmfels}, \cite{BHV}.  

\begin{theorem}\label{B-main}
There is a continuous map which identifies $\mathcal{T}(n)$ with a connected subcomplex of the analytification $X^{an}$. The evaluation map defined by the Pl\"ucker generators of $\C[X]$ takes $\mathcal{T}(n)$ isomorphically onto $\trop(I_{2, n})$. 
\end{theorem}

Representation theory provides many constructions which are  useful for construction of toric degenerations.  These methods underlie constructions used by  Alexeev and Brion \cite{Alexeev-Brion}, the Newton-Okounkov construction of Kaveh in \cite{Kaveh}, and the birational sequence approach used in \cite{Fang-Fourier-Littelmann-birational} , \cite{FFL} (see also \cite{Manon-Zhou}).  The construction of $\v_{\sigma, <}$ in Theorem \ref{main} relies instead on properties of the tensor product in the category of $\SL_2$ representations.  These valuations are also used in \cite{Manon-NOK}.   

Work on canonical bases in cluster algebras by Gross, Hacking, Keel, and Kontsevich \cite{GHKK}, and then later used by Rietsch and Williams \cite{RW} and Bossinger, Fang, Fourier, Hering, and Lanini \cite{BFFHL}on Grassmannians also provides a powerful organizing tool for toric degenerations.  In \cite{Kaveh-Manon-NOK}, \cite{GHKK}, and \cite{RW}, compactifications of varieties by nice divisors are linked with the construction of a toric degeneration.  We expect each compactification $X_\sigma$ can be realized via a potential function construction in the manner of \cite{GHKK} and \cite{RW}.  The compactification $X_\sigma$ is also closely related to the compactification of the free group character variety $\mathcal{X}(F_g, \SL_2)$ by a combinatorial normal crossings divisor constructed in \cite{Manon-Outer}.


\section{Background on valuations and tropical geometry}\label{valuations}

In this section we introduce the necessary background on filtrations of commutative algebras and the functions associated to these filtrations, valuations and quasi-valuations.   We recall the critical notions of adapted basis and Khovanskii basis for a valuation, which enable computations with valuations.  We also summarize results of Kaveh and the first author \cite{Kaveh-Manon-NOK}, which directly relate higher rank valuations to tropical geometry.

\subsection{Quasi-valuations and filtrations}

Let $A$ be a commutative domain over $\C$, and let $\Z^r$ be the free Abelian group of rank $r$ endowed with a total group ordering $\prec$ (e.g. the lexicographic ordering).  A (decreasing) algebraic filtration $F$ of $A$ with values in $\Z^r$ is the data of a $\C$-subspace $F_\alpha \subset A$ for each $\alpha \in \Z^r$ such that $F_\alpha \supset F_\beta$ when $\alpha \prec \beta$, $F_\alpha F_\beta \subset F_{\alpha + \beta}$, $\forall \alpha, \beta \in \Z^r$, and $\bigcup_{\alpha \in \Z^r} F_\alpha = A$.  We further assume that $1 \in F_0$ and $1 \notin F_\beta$ when $0 \prec \beta$.  For any $\alpha \in \Z^r$, we let $F_{\succ \alpha}$ be $\bigcup_{\beta \succ \alpha} F_\beta$.  For any such filtration, we can form the \emph{associated graded} algebra:

\begin{equation}
gr_F(A) = \bigoplus_{\alpha \in \Z^r} F_\alpha/F_{\succ \alpha}.\\
\end{equation}  
  
\begin{example}
If $A$ carries a $\Z^r$ grading, $A = \bigoplus_{\alpha \in \Z^r} A_\alpha$, then for any $\prec$ there is a filtration on $A$ defined by setting $F_\alpha = \bigoplus_{\beta \succeq \alpha} A_\beta$. In this case, $gr_F(A)$ is canonically isomorphic to $A$. 
\end{example}

Let $f \in F_\alpha \subset A$ but $f \notin F_{\succ \alpha}$, then we have the initial form $\bar{f} \in F_\alpha/F_{\succ \alpha} \subset gr_F(A)$.  It is straightforward to show that $\overline{fg} = \bar{f} \bar{g}$.  We say that $F$ only takes finite values if such an $\alpha$ exists for every $f \in A$.  We assume from now on that $F$ only takes finite values; this is the case for all the filtrations we consider in this paper.   

\begin{definition}
Let $F$ be a filtration as above. We define the associated quasi-valuation $\v_F: A \setminus \{0\} \to \Z^r$ as follows:\\

\begin{equation}
\v_F(f) = \alpha \ such \ that \ f \in F_\alpha, f \notin F_{\succ \alpha}.\\
\end{equation}

\end{definition}

The function $\v_F$ always has the following properties:\\

\begin{enumerate}
\item $\v_F(fg) \succeq \v_F(f) + \v_F(g)$,\\
\item $\v_F(f + g) \succeq MIN\{\v_F(f), \v_F(g)\}$,\\
\item $\v_F(C) = 0$, $\forall C \in \C$.\\
\end{enumerate}

\noindent
More generally, a function that satisfies $(1)-(3)$ above is called a \emph{quasi-valuation} on $A$.  If $\w:A \setminus \{0\} \to \Z^r$ is a quasi-valuation, we also get a corresponding filtration $F^\w$ defined as follows:\\

\begin{equation}
F^\w_\alpha = \{f \mid \w(f) \succeq \alpha\}.\\
\end{equation}

\noindent
One easily checks that the constructions $F \to \v_F$ and $\w \to F^\w$ are inverse to each other.   Finally, a quasi-valuation $\v$ is said to be a \emph{valuation} if $\v(fg) = \v(f) + \v(g)$, $\forall f, g \in A$.

\subsection{Adapted bases}

Now we recall the notion of an \emph{adapted basis} (\cite[Section 3]{Kaveh-Manon-NOK}).  Adapted bases facilitate computations and allow quasi-valuations to be treated as combinatorial objects.  We continue to use quasi-valuations with values in $\Z^r$, but we observe that the results in this section work with any ordered group. 

\begin{definition}\label{adapted}
A $\C$-vector space basis $\B \subset A$ is said to be \emph{adapted} to a filtration $F$ if $F_\alpha \cap \B$ is a vector space basis for all $\alpha \in \Z^r$. 
\end{definition}
 
If $\{F_\alpha \}$ is a collection of vector subspaces of $A$ (not necessarily forming a filtration) with the property that any intersection $F_\alpha \cap \B$ is a basis of $F_\alpha$, then the same property holds for any vector subspace of $A$ constructed by intersections and sums of the members of $\{F_\alpha\}$.  It immediately follows then that if the $F_\alpha$ form a filtration $F$, then $F_{\succ \alpha} \cap \B$ is a basis of $F_{\succ \alpha}$ and the equivalence classes $\bar{\B}_\alpha$ of basis members $\B \cap F_\alpha \setminus F_{\succ \alpha}$ form a basis of $F_\alpha/F_{\succ \alpha}$.  We let $\bar{\B} \subset gr_F(A)$ be the disjoint union $\sqcup_{\alpha \in \Z^r} \bar{\B}_\alpha$; this is a basis of $gr_F(A)$ which is adapted to the grading by $\Z^r$.   If a quasi-valuation $\v$ corresponds to a filtration $F$ with adapted basis $\B$, we say that $\B$ is adapted to $\v$.  The next proposition summarizes the basic properties of adapted bases. 

\begin{proposition}\label{adaptedbasisproperties}
Let $\v$ be a quasi-valuation with adapted basis $\B$, then:\\

\begin{enumerate}
\item for any $f \in A$ with $f = \sum_i C_i b_i$, $\v(f) = MIN\{\v(b_i) \mid C_i \neq 0\}$,\\
\item if $\B'$ is another basis adapted to $\v$, then any $b \in \B$ has a upper triangular expression in elements of $\B'$,\\
\item if $\B$ is adapted to another quasi-valuation $\w$ and $\v(b) = \w(b)$, $\forall b \in \B$, then $\v = \w$. \\
\end{enumerate}

\end{proposition}

\begin{proof}
For part $(1)$, let $\alpha = MIN\{\v(b_i) \mid C_i \neq 0\}$ and note that $f \in F_\alpha$, $\{b_i \mid C_i \neq 0\} \subset F_\alpha$, and $\v(f) \succeq \alpha$. If $f \in F_\beta$ with $\beta \succ \alpha$, then all $b_i \in F_\beta$, which is a contradiction.  
For part $(2)$, let $b \in \B$ have value $\v(b) = \alpha$; then we can write $b = \sum C_i b_i'$ for $b_i' \in \B'$. From part $(1)$ we have $\alpha = MIN\{\v(b_i') \mid C_i \neq 0\}$.  For part $(3)$,  note that for any $f\in A$ with $f = \sum C_i b_i$ we have $\v(f) = MIN\{\v(b_i) \mid C_i \neq 0\} = MIN\{\w(b_i) \mid C_i \neq 0\} = \w(f)$. 
\end{proof}

More generally, if we are given a direct sum decomposition $A = \bigoplus_{i \in I} A_i$ of $A$ as a vector space, we say that this decomposition is adapted to a filtration $F$ if for any $\alpha \in \Z^r$ we have $F_\alpha \cap A_i = A_i$ or $0$. Notice that in this case any selection of basis for each $A_i$ gives a basis adapted to $F$.  It is possible to define a sum operation on the set of all quasi-valuations adapted to a given basis $\B \subset A$.

\begin{definition}\label{valsum}
Let $\v, \w: A\setminus \{0\} \to \Z^r$ be quasi-valuations which share a common adapted decomposition $A = \bigoplus_{i \in I} A_i$, then we define the \emph{sum} $[\v + \w]: A \setminus \{0\} \to \Z^r$ to be $\v(b_i) + \w(b_i)$ for any $b_i \in A_i$.  We extend this to $f = \sum C_ib_i \in A$ by the setting $[\v+\w](f) = MIN\{[\v + \w](b_i) \mid C_i \neq 0\}$.
\end{definition}

\begin{proposition}\label{valsumprops}
The following holds for the sum operation:\\

\begin{enumerate}
\item the sum $\v + \w$ of two quasi-valuations is a quasi-valuation,\\
\item the sum operation is commutative and associative,\\ 
\item for any $\v$, $\sum_{i = 1}^n \v = n\v$,\\
\item the sum has neutral element, the quasi-valuation $\o: A \setminus \{0\} \to \Z^r$ defined by $\o(f) = 0$, $\forall f \in A\setminus \{0\}$,\\
\item the set of quasi-valuations adapted to $\B$ can be identified with the monoid of points in $[\Z^r]^{\B}$ which satisfy $\v(b_i) + \v(b_j) \preceq MIN\{\v(b_k) \mid b_ib_j = \sum C_k b_k, C_k \neq 0\}$.\\
\end{enumerate}
\end{proposition}

\begin{proof}
For part $(1)$, clearly $[\v + \w](C) = 0$, $\forall C \in \C$.  For $f, g \in A$, we write $f = \sum C_i b_i$ and $ g= \sum K_i b_i$ for elements $b_i \in A_i$, so that $\v(f) = MIN\{ \v(b_i) \}$ and $\v(g) = MIN\{ \v(b_j)\}$.  The sum $[\v +\w](f + g)$ is then computed  by $MIN\{ \v(b_i) +\w(b_i) \mid C_i + K_i \neq 0\}$; 
this must be larger than $MIN\{ \v(b_i) +\w(b_i) \mid C_i \neq 0\}$ and $MIN\{ \v(b_i) +\w(b_i) \mid K_i \neq 0\}$.  Now, consider the product $fg = \sum C_iK_j b_i b_j$; this is a sum $fg = \sum T_i b_i$ with the property that $\v(f) + \v(g) \preceq MIN\{ \v(b_i) \mid T_i \neq 0\}$  and $\w(f) + \w(g) \preceq MIN\{ \w(b_i) \mid T_i \neq 0\}$.  We have $[\v + \w](f) + [\v + \w](g) \preceq MIN\{ \v(b_i) \mid T_i \neq 0\} + MIN\{ \w(b_i) \mid T_i \neq 0\} \preceq MIN\{ \v(b_i) + \w(b_i) \mid T_i \neq 0\} = [\v + \w](fg)$.

For Parts $(2)-(4)$, we observe that this operation is commutative and associative by definition.  It is easy to check $[\sum_{i = 1}^n \v](b_i) = n\v(b_i)$ for any $b_i \in A_i$; this implies that $[\sum_{i = 1}^n \v](f) = n\v(f)$ for any $f \in A$.  Similarly, $[\v + \o](b_i) = \v(b_i)$ for any $b_i \in A_i$; this implies that $\o$ is a neutral element.  For part $(5)$, we leave it to the reader to consider the map which sends $\v$ to the tuple $(\v(b) \mid b \in \B) \in [\Z^r]^{\B}$. 
\end{proof}

Proposition \ref{valsumprops} illustrates how quasi-valuations with a common adapted basis tend to work well with each other. The following lemma shows a similar phenomenon. 

\begin{lemma}\label{adaptedquasi}
If $\v_1, \v_2$ are quasi-valuations which are adapted to the same basis $\B \subset A$, then the function $\bar{\v}_2$ on $gr_{\v_1}(A)$
which assigns $\v_2(b)$ to $\bar{b} \in \bar{\B} \subset gr_{\v_1}(A)$ also defines a quasi-valuation.  
\end{lemma}

\begin{proof} We only have to check that we have $\bar{\v}_2(\bar{b}_i\bar{b}_j)\succeq \bar{\v}_2(\bar{b}_i) + \bar{\v}_2(\bar{b}_j)$. Now,  $\bar{\v}_2(\bar{b}_i\bar{b}_j) = MIN\{ \bar{\v}_2(\bar{b}_k) \mid \bar{b}_i \bar{b}_j = \sum C_k\bar{b}_k, C_k \neq 0\}$. But the equation $\bar{b}_i \bar{b}_j = \sum C_k\bar{b}_k$ is a truncation of the corresponding expansion of $b_ib_j$ in $A$, where the associated inequality holds, that is, $\v_2(b_i) + \v_2(b_j) \preceq MIN\{\v_2(b_k) \mid b_ib_j = \sum C_k b_k, C_k \neq 0\}$.
\end{proof}

The sum operation on quasi-valuations is easy to work with when dealing with tensor products of algebras.   Let $\v_1$ be a quasi-valuation on $A_1$ and $\v_2$ be a quasi-valuation on $A_2$, and let $F^1, F^2$ be the corresponding filtrations.  We get two filtrations $\mathcal{F}^1,\mathcal{F}^2$ on $A_1 \otimes_\C A_2$ by setting $\mathcal{F}^1_\alpha = F^1_\alpha \otimes A_2$ and $\mathcal{F}^2_\alpha = A_1 \otimes F^2_\alpha$, with corresponding quasi-valuations $\v_1, \v_2$. By picking adapted bases (this is always possible for the algebras we consider in this paper) $\B_1$ and  $\B_2$ we obtain a basis $\B = \{ b_i \otimes b_j \mid b_i \in \B_1, b_j \in \B_2\} \subset A_1 \otimes_\C A_2$ which is simultaneously adapted to $\v_1$ and $\v_2$. 

\begin{lemma}\label{tensorsum}
For $\B, \v_1$, and $\v_2$ as above,  we have $gr_{\v_1 + \v_2}(A_1\otimes_\C A_2) \cong gr_{\v_1}(A_1) \otimes_\C gr_{\v_2}(A_2)$. Moreover, $\v_1 + \v_2$ is independent of the choice of bases $\B_1$ and  $\B_2$. 
\end{lemma}

\begin{proof}
Clearly as vector spaces we have $gr_{\v_1 + \v_2}(A_1\otimes_\C A_2) \cong gr_{\v_1}(A_1) \otimes_\C gr_{\v_2}(A_2)$. So, it remains
to show that the multiplication operations on both sides coincide.  This follows from the fact that $\v_1$ only sees the first tensor component, and $\v_2$ only sees the second tensor component; this in turn implies that the lower terms of the product  $[b_1\otimes b_1'][b_2\otimes b_2'] = [b_1b_2 \otimes b_1'b_2']$ are the same way on both sides.  We leave the second statement to the reader. 
\end{proof}

Finally, we will need the following notion of \emph{Khovanskii basis}.

\begin{definition}[Khovanskii basis]
We say $\mathcal{B} \subset A$ is a Khovanskii basis for a quasi-valuation $\v$ if $gr_\v(A)$ is generated by the equivalence
classes $\bar{\mathcal{B}} \subset gr_\v(A)$ as an algebra over $\C$. 
\end{definition}

\subsection{Weight valuations and the tropical variety}\label{weightvaluation}

For the following construction see \cite[Section 4]{Kaveh-Manon-NOK}. We assume that $A$ is presented as the image of a polynomial ring: $\pi: \C[\bx] \to A$, with kernel $Ker(\pi) = I$. Here $\bx = \{x_1, \ldots, x_n\}$ is a system of parameters.  We make the further assumption that $A$ is a positively graded domain.  Recall the notion of initial form $in_w(f)$ of a polynomial $f \in \C[\bx]$ and initial ideal $in_w(I)$ associated to a weight vector $w \in \Q^n$.  We will require the notion of the Gr\"obner fan $\mathcal{G}(I)$ associated to $I$; recall that this is a  complete polyhedral fan in $\Q^n$ whose cones index the initial ideals of $I$.  In particular, we have $in_w(I) = in_{w'}(I)$ for any $w, w'$ which are members of the relative interior of the same cone in $\mathcal{G}(I)$.  For this and other notions from Gr\"obner theory see \cite{GBCP} and \cite{Maclagan-Sturmfels}.

The tropical variety $\trop(I)$ can be identified with a subfan of $\mathcal{G}(I)$ given by those cones whose associated initial ideals contain no monomial (see \cite{Speyer-Sturmfels}, \cite{Maclagan-Sturmfels}).   The tropical variety $\trop(I)$, and more generally the Gr\"obner fan $\mathcal{G}(I)$ of the ideal $I$ help to organize the quasi-valuations on $A$ with Khovanskii basis $\pi(\bx) = \mathcal{B}$ by realizing all such functions as so-called weight quasi-valuations.

\begin{definition}[Weight quasi-valuations]
For $w \in \Q^n$ the weight quasi-valuation on $A = \C[\bx]/I$ is defined on $f \in A$ as follows:

\begin{equation} 
\v_w(f) = MAX\{MIN\{ \langle w, \alpha \rangle | \ p(\bx) = \sum C_\alpha \bx^\alpha, C_\alpha \neq 0\} \mid \pi(p) = f\}
\end{equation}
\end{definition}

We summarize the properties of weight quasi-valuations that we will need in the following proposition (see \cite[Section 4]{Kaveh-Manon-NOK}). We let $gr_w(A)$ denote the associated graded algebra of $\v_w$. 

\begin{proposition}\label{mainKhovanskii}
Let $A$ be a positively graded algebra presented as $\C[\bx]/I$ for a prime ideal $I$, then:\\

\begin{enumerate}
\item for any $w \in \Q^n$, $gr_w(A) \cong \C[\bx]/in_w(I)$,\\
\item $\v_w$ is adapted to any standard monomial basis of $A$ associated to a monomial ordering on $I \subset \C[\bx]$ which refines $w$,\\
\item $\v: A \setminus \{0\} \to \Q$ is a quasi-valuation with Khovanskii basis $\mathcal{B} = \pi(\bx)$ if and only if $\v = \v_w$ for some $w \in \Q^n$. 
\end{enumerate}

\end{proposition}

If $in_w(I)$ is a prime ideal, then part $(1)$ of Proposition \ref{mainKhovanskii} implies that $\v_w$ is a valuation.  In this case we say that the cone $C_w$ of the Gr\"obner fan containing $w$ in its relative interior is a \emph{prime cone}.  With a mild assumption (each element of $\bx$ is a standard monomial), we can conclude that $C_w \subset \trop(I)$.

\section{Constructions for $\SL_2$ and $\A^2$}\label{sl2}

In this section we define compactifications of $\SL_2$ and $\A^2$ which are stable under the group actions on these spaces (respectively by $\SL_2 \times \SL_2$ and $\SL_2$).  The divisorial valuations defined by the boundaries of these compactifications are used as building blocks in both the tropical and Newton-Okounkov constructions we give for the Grassmannian variety, and the compactifications themselves are key ingredients in the construction of the projective variety $X_\sigma$.  Accordingly, the constructions presented here for $\SL_2$ and $\A^2$ provide a reference point for the main results of the paper.  

\subsection{Representations of $\SL_2$}

Recall that $\SL_2$ is a simple algebraic group over $\C$.  This implies that any finite dimensional representation $V$ of $\SL_2$ decomposes uniquely into a direct sum of irreducible representations:

\begin{equation}
V \cong \bigoplus_{n \geq 0} Hom_{\SL_2}(V(n), V)\otimes V(n).\\
\end{equation}

\noindent
The representation $V(n)$ is the irreducible representation of $\SL_2$ associated to the dominant weight $n \in \Z_{\geq 0}$.  The representation $V(n)$ is isomorphic to the $n$-th symmetric power $Sym^n(\C^2)$; in particular $V(0)$ is isomorphic to $\C$ equipped with the trivial action by $\SL_2$. The vector space $Hom_{\SL_2}(V(n), V)$ is the space of $\SL_2$-maps from the irreducible $V(n)$ into $V$, which is called the multiplicity space of $V(n)$ in $V$.  The space $Hom_{\SL_2}(V(0), V)$ is called the space of $\SL_2$-invariants in $V$, which is also denoted by $V^{\SL_2}$.

For any two $\SL_2$-representations $V$ and  $W$, we can consider the tensor product $V\otimes W$ equipped with the diagonal action $g\circ (v\otimes w) = g\circ v \otimes g \circ w$.  Similarly, the vector space of homomorphisms $Hom(V, W)$ is naturally equipped with a representation structure; in particular, the dual vector space $V^* = Hom(V, V(0))$ is called the dual representation.  For any $n \in \Z_{\geq 0}$ we have $V(n)^* \cong V(n)$. 
 These operations endow the category $Rep(\SL_2)$ of finite dimensional $\SL_2$-representations with the structure of a symmetric, monoidal, semi-simple category with dualizing object $V(0)$.  It is an important problem for any such category to determine the rule for decomposition of a tensor product of irreducible representations into irreducibles:

\begin{equation}
V(j)\otimes V(k) = \bigoplus_{i \geq 0} Hom_{\SL_2}(V(i), V(j)\otimes V(k))\otimes V(i).\\
\end{equation}

We have $Hom_{\SL_2}(V(i), V(j)\otimes V(k)) \cong Hom_{\SL_2}(V(0), V(i)^*\otimes V(j) \otimes V(k)) \cong [V(i)\otimes V(j) \otimes V(k)]^{\SL_2}$ using the properties of tensor product and duals, so this problem can be reduced to computing the invariant spaces $[V(i)\otimes V(j) \otimes V(k)]^{\SL_2}$.  The following formula can be derived from the \emph{Pieri rule}, \cite[6.1]{Fulton-Harris}:

\begin{equation}\label{pieri}
[V(i)\otimes V(j) \otimes V(k)]^{\SL_2} \cong \Big\{ \begin{array}{ll} \C & \text{if }  i + j + k \in 2\Z, \ \  |i - j| \leq k \leq i + j, \\ 0 & \text{otherwise}\\ \end{array} 
\end{equation}

\noindent
We refer to $i + j + k \in 2\Z$ as the \emph{parity condition} on a triple of integers.  We say that $(i, j, k)$ satisfy the \emph{triangle inequalities} if $0 \leq i, j, k$ and $ |i - j| \leq k \leq i + j$; this is because these are precisely the conditions needed to guarantee that $i, j, k$ can be the sides of a Euclidean triangle.

\subsection{Coordinate algebras of $\SL_2$ and $\A^2$}

Recall the isotypical decomposition of the coordinate ring of $\SL_2$ as an $\SL_2 \times \SL_2$-representation:

\begin{equation}
\C[\SL_2] = \bigoplus_{n \geq 0} V(n) \otimes V(n).\\
\end{equation}

\noindent
The multiplication operation $m: \C[\SL_2] \otimes \C[\SL_2] \to \C[\SL_2]$ is not graded by dominant weight, but the dominant weights still define a filtration. For any $n$ and $m \in \Z_{\geq 0}$ we have:

\begin{equation}\label{lowermult}
m \Big( [V(m)\otimes V(m)] \otimes [V(n) \otimes V(n)]\Big) \subset \bigoplus_{k \leq n +m} V(k) \otimes V(k)
\end{equation}

\noindent
In particular, the projection of $m \Big( [V(m)\otimes V(m)] \otimes [V(n) \otimes V(n)]\Big)$ onto $V(n+m)\otimes V(n+m)$ is an instance of the so-called \emph{Cartan multiplication} operation on tensor products of irreducible representations, and is never $0$ (see \cite[Section 3]{HMM}).  There is an algebraic filtration of $\C[\SL_2]$ by the spaces:

\begin{equation}
F_m = \bigoplus_{n \leq m} V(n) \otimes V(n).\\
\end{equation}

\noindent
Using Equation \ref{lowermult}, it is straightforward to check that $m(F_m \otimes F_n) \subset F_{m+n}$.

Let $U \subset \SL_2$ be the group of upper triangular $2\times 2$ matrices with $1$'s along the diagonal.  Using right multiplication by elements of $U$, any element of $\SL_2$ can be taken to a matrix whose entries depend only on the two entries in the first column.  Since both of these entries cannot be zero, we find that $\SL_2/U \cong \A^2 \setminus \{0\}$.   Since the origin is a codimension-$2$ subvariety of $\A^2$, we have an isomorphism of the algebra of $U$-invariants $\C[\SL_2]^U$ with the coordinate ring of $\A^2$; namely a polynomial ring on two variables. 

The group $U$ acts on the right hand component of each tensor product $V(n) \otimes V(n) \subset \C[\SL_2]$.  As each $V(n)$ is irreducible, with a one-dimensional subspace of highest weight vectors, the space $V(n)^U$ has dimension $1$,  so $V(n)\otimes V(n)^U \cong V(n)$.  It follows that $\C[\A^2] = \C[SL_2]^U$ has the following isotypical decomposition:

\begin{equation}\label{isoA2}
\C[\A^2] = \bigoplus_{n \geq 0} V(n).\\
\end{equation} 

\noindent
Indeed, $V(n) \cong Sym^n(\C^2)$, so equation \ref{isoA2} is the direct sum decomposition of the polynomial ring on two variables into its homogeneous components.   The multiplication operation on $\C[\A^2]$, just normal polynomial multiplication, is accordingly the Cartan multiplication operation for $\SL_2$: $V(n) \otimes V(m) \to V(n + m)$.  This grading endows $\A^2$ with an action by $\G_m$ on the right in addition to its natural action by $\SL_2$ on the left. In particular, $t \in \G_m$ acts on $f \in V(n)$ by the rule $f \circ t = ft^n$.

The associated graded algebra $gr_F(\C[\SL_2])$ of the filtration $F$ has an identical isotypical decomposition to $\C[\SL_2]$,

\begin{equation}
gr_F(\C[\SL_2]) = \bigoplus_{n \geq 0} V(n) \otimes V(n).\\
\end{equation}

\noindent
The difference between these two algebras is found in their multiplication operations, where the multiplication in $gr_F(\C[\SL_2])$ is computed by the Cartan multiplication operation.  Following \cite[Section 3]{HMM} and \cite{Popov}, we say that $gr_F(\C[\SL_2])$ is the coordinate algebra of the horospherical contraction $\SL_2^c$ of $\SL_2$.   The coordinate ring $\C[\SL_2^c]$ can also be constructed by means of invariant theory.  We have $\G_m$ act antidiagonally through the right actions on two copies of the coordinate ring of $\A^2$.  In particular, for $t \in \G_m$ and $f \in V(n)\otimes V(m) \subset \C[\A^2] \otimes \C[\A^2]$ we have $f \circ t = ft^{m-n}$. The only components which are invariant under this action are those with $m = n$. The coordinate ring of the horospherical contraction $\SL_2^c$ can be constructed by taking invariants with respect to this action:

\begin{equation}
\C[\SL_2^c] =  \Big[ \C[\A^2] \otimes \C[\A^2] \Big]^{\G_m}.\\
\end{equation}

\subsection{Valuations on $\C[\SL_2]$ and $\C[\A^2]$}

The algebra $\C[\SL_2^c]$ is a domain, so it follows that the filtration $F$ defines a valuation $v: \C[\SL_2]\setminus \{0\} \to \Z$.  This valuation is computed on a regular function $f \in \C[\SL_2]$ with $f = \sum f_n$, $f_n \in V(n)\otimes V(n)$, by the rule:

\begin{equation}
v(f) = MIN\{ -n \mid f_n \neq 0\}. \\ 
\end{equation}

\noindent
Abusing notation, we say that $\C[\SL_2^c]$ is the associated graded algebra of $v$.  Likewise, the algebra $\C[\A^2]$ is equipped with its degree valuation $deg: \C[\A^2]\setminus\{0\} \to \Z$, which  is computed using almost the same formula; for $f \in \C[\A^2]$ with $f = \sum f_n$, $f_n \in V(n)$, we have $deg(f) = MIN\{-n \mid f_n \neq 0\}$.  Notice that this is the \emph{negative} of the homogeneous degree function on $\C[\A^2]$.  Where $v$ is an $\SL_2 \times \SL_2$-invariant valuation on $\C[\SL_2]$, $deg$ is invariant with respect to the action of $\SL_2 \times \G_m$ on $\A^2$.  This will feature prominently in our constructions involving the Pl\"ucker algebra.  

Now we define the \emph{Rees algebra} of the valuation $v$:

\begin{equation}
R = \bigoplus_{m \geq 0} F_mt^m = \bigoplus_{m \geq n \geq 0} V(n)\otimes V(n)t^m.\\
\end{equation}

\noindent
The parameter $t \in V(0)\otimes V(0)t \subset F_1t$ acts by ``shifting" the copy $V(n)\otimes V(n)t^m$ of the space $V(n)\otimes V(n) \subset F_m$ to the copy of the same space $V(n)\otimes V(n)t^{m+1} \subset F_{m+1}t^{m+1}$.  Since $t$ is not a $0$-divisor, this action makes $R$ into a flat $\C[t]$-module. For the following see \cite[Section 3]{HMM}. 

\begin{lemma}\label{ReesFactsSL2}
The following hold for the $\C[t]$ action on $R$.\\ 

\begin{enumerate}
\item $\frac{1}{t}R \cong \C[\SL_2]\otimes \C[t, t^{-1}]$,\\
\item $R/tR \cong gr_F(\C[\SL_2]) \cong \C[\SL_2^c]$. \\
\end{enumerate}
\end{lemma}

\noindent
Part $(1)$ of Lemma \ref{ReesFactsSL2} says that away from the origin we have $R/(t-a)R \cong \C[SL_2]$, whereas part $(2)$ says at the special fiber $R/tR$ we obtain $\C[SL_2^c]$.  

In coordinates $\C[\SL_2] \cong \C[a, b, c, d]/\langle ad-bc -1\rangle$ for $a, b, c, d \in V(1)\otimes V(1)$.  Cartan multiplication must be surjective (the image is irreducible), so it follows that $a, b, c, d \in V(1)\otimes V(1)$ generate $\C[\SL_2^c]$ as well.  Picking coordinates $V(1) \cong \C\{x, y\}$ we can set $a = x \otimes x$, $b = x \otimes y$, $c = y\otimes x$, and $d = y \otimes y$ (i.e. these are the ``matrix entries" of a $2\times 2$ matrix).  Computing in $\C[\SL_2^c]$ we see that $ad - bc = (x\otimes x)(y \otimes y) - (x \otimes y)(y \otimes x)$.  In the coordinate ring $\C[\A^2]\otimes \C[\A^2]$ this is $(xy - xy) \otimes (xy - yx) = 0$.  It follows that we can identify $\SL_2^c$ with the singular $2\times 2$ matrices.  If we set $A = at, B = bt, C = ct, D = dt \in R$ we can compute $AD - BC - t^2 = 0$; this defines a presentation of $R$.  Passing from a general point ($t \neq 0$) to the origin ($t = 0$) degenerates $\SL_2$ to $\SL_2^c$: the singular $2\times 2$ matrices. 

\begin{remark}
In \cite{HMM} and \cite{Manon-Outer}, a different Rees family is used.  Instead of $AD - BC - t^2$, the family is presented by $AD-BC - s$, where $s$ is a parameter of homogeneous degree $2$.  In this way, the family we consider, $Spec(R)$, is a double cover of the family cut out by $AD-BC -s$ considered in \emph{loc. cit.}. 
\end{remark}

\subsection{Compactifications}

Now we define a compactification of $\SL_2$ by setting $\overline{\SL}_2 = Proj(R)$.  

\begin{proposition}\label{sl2geometry}
The following are true of the projective scheme $\overline{\SL}_2$:\\
\begin{enumerate}
\item $\overline{\SL}_2$ has an algebraic action by $\SL_2 \times \SL_2$,\\
\item $\overline{\SL}_2$ can be identified with the closed subscheme of $\P^4$ cut out by $AD- BC - t^2 = 0$,\\
\item the $\SL_2\times \SL_2$-stable irreducible divisor $D \subset \overline{\SL}_2$ defined by setting $t = 0$ is isomorphic to $\P^1\times \P^1$,\\
\item $\SL_2$ is isomorphic to the Zariski-open complement of $D$, \\
\item the line bundle $\O(1)$ defined by the divisor $D$ satisfies $H^0(\overline{\SL}_2, \O(m)) \cong F_m$. Furthermore, this line bundle induces $\O(1)\boxtimes\O(1)$ on $D \cong \P^1\times\P^1$,\\
\item the valuation $ord_D: \C[\SL_2] \setminus \{0\} \to \Z$ is equal to $v$. \\
\end{enumerate}

\end{proposition}

\begin{proof}
This is essentially contained in \cite{Manon-Outer}, but we will also give a proof here.  Part $(1)$ follows from the definition of $\overline{\SL}_2$ as $Proj$ of an $\SL_2\times \SL_2$-algebra.  Similarly, parts $(2)$, $(3)$, $(4)$, and $(5)$ follow from the presentation of $R$ given above.  For part $(6)$, we identify $SL_2$ with the open subset $Spec([\frac{1}{t}R]_0) \subset \overline{\SL}_2$. The role of $t$ as a placeholder in the direct sum decomposition of the Rees algebra makes the use of ``$t$'' in this description of $\SL_2$ misleading; to be precise we refer to the regular function $1t \subset V(0)\otimes V(0)t^1$.  Taking $ord_D$ of a regular function measures divisibility by $1t$, so we will determine what degree of $1t$ divides an element $f \in V(n)\otimes V(n)$.   In order to be in the degree-$0$ part of $\frac{1}{t}R$, we must divide $V(n)\otimes V(n)t^m$ by $(1t)^m$ to obtain $\frac{1}{(1t)^m}[V(n)\otimes V(n)t^m]$. However every function in this component is already divisible by $(1t)^{m-n}$, so we obtain $\frac{1}{(1t)^n}[V(n)\otimes V(n)t^n]$; this is the component which maps to $V(n)\otimes V(n)$ under the isomorphism $[\frac{1}{t}R]_0 \cong \C[\SL_2(\C)]$.  It follows that $ord_D(f) = -n$ for any $f \in V(n) \otimes V(n) \subset \C[\SL_2]$. Since $D$ is $\SL_2 \times \SL_2$-invariant, the valuation $ord_D$ is as well; as a consequence (see \cite[Chapter 4]{Timashev}) we compute $ord_D(f)$ for $f = \sum f_n$, $f_n \in V(n) \otimes V(n)$,  by taking $MIN\{ord_D(f_n) \mid f_n \neq 0\}$.   
\end{proof}

A similar statement holds for $\A^2$.  We  form the Rees algebra $S = \bigoplus_{m \geq n \geq 0} V(n)t^m$ with respect to the valuation $deg$, and take $Proj(S)$ to obtain the $\SL_2\times \G_m$-stable compactification $\A^2 \subset \P^2$.  The divisor at infinity in this compactification is $Proj(\C[\A^2]) \cong \P^1$. The sections of this divisor recover $\O(1)$ on both $\P^2$ and the boundary $\P^1$. The valuation computed by taking order along the boundary recovers the degree valuation $deg: \C[\A^2] \setminus \{0\} \to \Z$.

\section{Construction of $X$ and $X_\sigma$}\label{GIT}

In this section we describe a construction of the affine cone $X$ over the Pl\"ucker embedding of the Grassmannian variety $\Gr_2(\C^n)$ which depends on the choice of a tree $\sigma$ with $n$ labeled leaves.  This construction uses aspects of the geometry of $\SL_2$ and $\A^2$ described in Section \ref{sl2}.   We obtain a compactification $X_\sigma \supset X$ by performing the same construction with the compactifications $\overline{\SL}_2 \supset \SL_2$ and $\P^2 \supset \A^2$. In this section we make frequent use of the language of Geometric Invariant Theory (GIT). For background on this subject see the book of Dolgachev: \cite{Dolgachev}.

\subsection{Constructing the affine cone $X$ from a tree $\sigma$}

We fix a tree $\sigma$ with $n$ leaves labeled by $i \in [n]$ with a cyclic ordering $i_1 \to \cdots \to i_n \to i_1$.  Let $V(\sigma)$ be the set of non-leaf vertices of $\sigma$, and $E(\sigma)$ be the set of edges of $\sigma$.  We further define $L(\sigma)$ to be the set of  leaf-edges of $\sigma$, i.e. those edges which connect to a leaf, and $E^{\circ}(\sigma)$ to be the set of   non-leaf edges.  In particular, we have $E(\sigma) = E^{\circ}(\sigma) \sqcup L(\sigma)$.  We let $\ell_i \in L(\sigma)$ denote the leaf-edge which is connected to the leaf labeled $i$.  

We select an  orientation on $\sigma$; in particular, we choose a direction on each $e \in E(\sigma)$ so that the head of $\ell_i \in L(\sigma)$ points toward the leaf $i$.  This information is necessary to construct $X$ and $X_\sigma$, but ultimately the construction is independent of this choice.

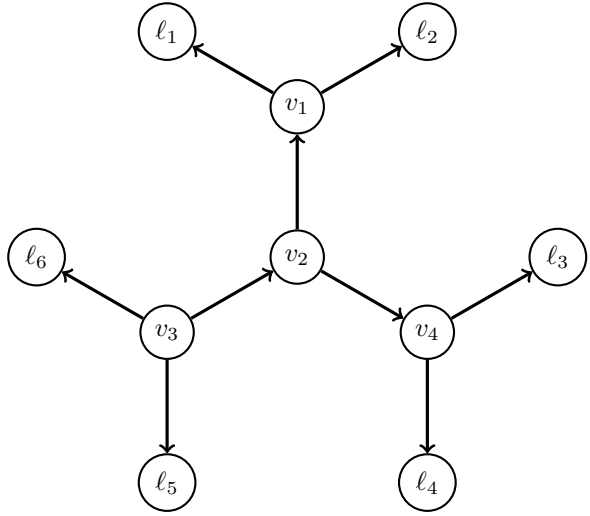
\begin{figure}
\begin{tikzpicture}[
every edge/.style = {draw=black,very thick, ->},
 vrtx/.style args = {#1/#2}{%
      circle, draw, thick, fill=white,
      minimum size=5/2mm, label=#1:#2}]
\node(A) [vrtx=left/] at (0, 0) {$v_2$};
\node(B) [vrtx=left/] at (0,4/2) {$v_1$};
\node(C) [vrtx=left/] at (-3.46/2,-2/2) {$v_3$};
\node(D) [vrtx=left/] at (3.46/2, -2/2) {$v_4$};
\node(E) [vrtx=left/] at (-3.46/2,6/2){$\ell_1$};
\node(F) [vrtx=left/] at (3.46/2,6/2){$\ell_2$};
\node(G) [vrtx=left/] at (-6.93/2,0){$\ell_6$};
\node(H) [vrtx=left/] at (-3.46/2,-6/2){$\ell_5$};
\node(I) [vrtx=left/] at (6.93/2,0){$\ell_3$};
\node(J) [vrtx=left/] at (3.46/2,-6/2){$\ell_4$};
\path   (A) edge (B)
           (C) edge (A)
	(A) edge (D)
	(B) edge (E)
	(B) edge (F)
	(C) edge (G)
	(C) edge (H)
	(D) edge (I)
	(D) edge (J);
\end{tikzpicture}
\caption{An oriented tree $\sigma$.}
\end{figure}

We define a space $M(\sigma)$ and an algebraic group $G(\sigma)$ using elements of the tree $\sigma$.  The space $M(\sigma)$ is a product of copies of $\SL_2$ and $\A^2$, with one copy of $\SL_2$ for each non-leaf edge, and one copy of $\A^2$ for each leaf-edge:

\begin{equation}
M(\sigma) = \prod_{e \in E^{\circ}(\sigma)} \SL_2 \times \prod_{\ell \in L(\sigma)} \A^2.\\
\end{equation}

\noindent
Similarly, the group $G(\sigma)$ is a product of copies of $\SL_2$, with one copy of $\SL_2$ for each non-leaf vertex:

\begin{equation}
G(\sigma) = \prod_{v \in V(\sigma)} \SL_2.\\
\end{equation}

\noindent
Now we define an action of $G(\sigma)$ on $M(\sigma)$.  For a non-leaf vertex $v \in V(\sigma)$, we have the corresponding copy of $\SL_2 \subset G(\sigma)$ act on the left hand side of the space assigned to an out-going edge, and on the right hand side of any incoming edge.  Notice that leaf-edges are always assigned a copy of $\A^2$ which comes with an action by $\SL_2 \times \G_m$ as described in Section \ref{sl2}; so for any vertex $v$ connected to a leaf-edge the corresponding copy of $\SL_2$ acts on the left hand side of $\A^2$ by our conventions.

\begin{figure}
\begin{tikzpicture}[
every edge/.style = {draw=black,very thick, ->},
 vrtx/.style args = {#1/#2}{%
      circle, draw, thick, fill=white,
      minimum size=5/2mm, label=#1:#2}]
\node(A) [vrtx=left/] at (0, 0) {};
\node(B) [vrtx=left/] at (0,4/2) {$\SL_2$};
\node(C) [vrtx=left/] at (-3.46/2,-2/2) {$\SL_2$};
\node(D) [vrtx=left/] at (3.46/2, -2/2) {$\SL_2$};
\node(E) [vrtx=left/] at (-3.46/2,6/2){};
\node(F) [vrtx=left/] at (3.46/2,6/2){};
\node(G) [vrtx=left/] at (-6.93/2,0){};
\node(H) [vrtx=left/] at (-3.46/2,-6/2){};
\node(I) [vrtx=left/] at (6.93/2,0){};
\node(J) [vrtx=left/] at (3.46/2,-6/2){};
\node(L) at (.5, 1){$\SL_2$};
\node() at (-1, -.1){$\SL_2$};
\node() at (.8, -.9){$\SL_2$};
\node() at (3.46/2+.5, -2){$\A_2$};
\node() at (-3.46/2+.5, -2){$\A_2$};
\node() at (-1,2){$\A_2$};
\node() at ((-3.46/2 -1, -2/2){$\A_2$};
\node() at ((3.46/2 +.7,-.1){$\A_2$};
\node() at ((1.1, 3.3,){$\A_2$};
\path   (A) edge (B)
           (C) edge (A)
	(A) edge (D)
	(B) edge (E)
	(B) edge (F)
	(C) edge (G)
	(C) edge (H)
	(D) edge (I)
	(D) edge (J);
\end{tikzpicture}
\caption{The space $M(\sigma)$ with action by $G(\sigma)$.}
\end{figure}
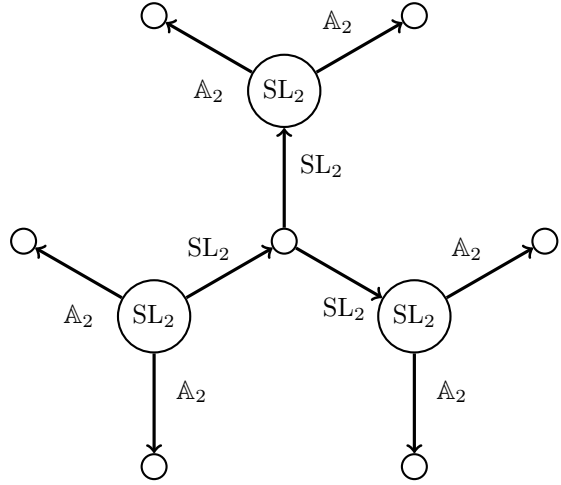

\begin{proposition}\label{GITXspace}
For any tree $\sigma$ with $n$ labeled leaves, the GIT quotient $M(\sigma)\qr G(\sigma)$ is isomorphic to $X$. 
\end{proposition}

\begin{proof}
It is well-known (see \cite{Dolgachev}) that $X$ can be constructed as the GIT quotient $\SL_2 \ql [\A^2 \times \cdots \times \A^2]$; this is equivalent to the fact that the Pl\"ucker algebra is generated by the $2\times2$ minors of a $2\times n$ matrix of parameters.  This quotient can be recovered from the GIT construction above as the case of a tree $\sigma_n$ with $n$ labeled leaves, one non-leaf vertex, and the natural cyclic ordering $1 \to \cdots \to n$. Therefore, to prove the proposition it suffices to show that all of the GIT constructions are isomorphic to $M(\sigma_n)\qr G(\sigma_n)$.  The cyclic ordering does not affect the isomorphism type, so the problem can be reduced to showing the following statement: for any tree $\sigma$ as above, and a tree $\sigma'$ obtained from $\sigma$ by contracting an edge $e \in E^{\circ}(\sigma)$, we have $M(\sigma)\qr G(\sigma) \cong M(\sigma')\qr G(\sigma')$.  

Geometric invariant theory quotients can be performed in stages, so we can further reduce to the case of the trees $\sigma$ with only one non-leaf edge $e$, and $\sigma'$ with no non-leaf edges. Moreover, as $\SL_2\qr U = \A^2$ (see Section \ref{sl2}), we may once again invoke GIT-in-stages to assume each edge of $\sigma$ and $\sigma'$ has been assigned a copy of $\SL_2$.   Let $e \in E^{\circ}(\sigma)$ have vertices $v_1, v_2$, with the orientation along $e$ pointing $v_1 \to v_2$.  Let $v_1$ have leaf edges $\ell_1, \ldots, \ell_s$ and $v_2$ have leaf edges $k_1, \ldots, k_r$.  We orient $\ell_1, \ldots, \ell_s$ and $k_1, \ldots, k_r$ away from $v_1, v_2$. We make this choice without a loss of generality as we have the involution $g \to g^{-1}$, which is an isomorphism on the scheme $\SL_2$ which interchanges the left and right actions.  We let $v \in V(\sigma')$ be the lone non-leaf vertex of $\sigma'$, and by abuse of notation we let $\ell_1, \ldots, \ell_s$ and $k_1, \ldots, k_r$ be its leaf-edges, oriented in the same fashion. 

As above, we define $M(\sigma) = (\prod_{\ell_i} \SL_2) \times \SL_2 \times (\prod_{k_j} \SL_2)$ with an action of $G(\sigma) = \SL_2\times\SL_2$.  Similarly, 
$M(\sigma') = \prod_{\ell_i, k_j} \SL_2$ with an (now entirely left) action by $\SL_2$.   We claim that there is an isomorphism $M(\sigma)\qr \big( \SL_2\times\SL_2 \big)\cong M(\sigma')\qr \SL_2$.  Once more, we appeal to GIT-in-stages and show that $\big(\prod_{\ell_i} \SL_2 \big)\times \SL_2 \qr \SL_2 \cong \prod_{\ell_i}\SL_2 $ as spaces with an action by $\SL_2$.  Here $\SL_2$ acts on the right hand side of the second component of $\big(\prod_{\ell_i} \SL_2 \big)\times \SL_2$, and on the left hand sides of the components of $\prod_{\ell_i}\SL_2$.  

To prove this we show something more general.  Let $X$ be a $G$-variety for a reductive group $G$, and let $G$ act on $X\times G$ diagonally on $X$ and the left hand side of $G$, then $X\times G \qr G$ retains an action of $G$ through the right hand side of $G$ in $X\times G$.  As $G$-varieties we have $X\times G\qr G \cong X$.  To show this, map $(x, g) \in X \times G$ to $g^{-1}x \in X$; this is a map of $G$-spaces which intertwines the right action on $G$ in $X\times G$ with the action on $X$. This map is constant on the orbits of $X\times G$ under the diagonal action, which are in turn all closed; and furthermore there is an algebraic section $X \to X\times G$ sending $x$ to $(x, Id)$ for $Id \in G$ the identity. This proves the result. 
\end{proof}

By Proposition \ref{GITXspace}, each tree $\sigma$ defines a different realization of $X = \SL_2 \ql[ \A^2 \times \cdots \times \A^2]$ with added ``hidden variables" given by the $\SL_2$ components along the non-leaf edges.  The combinatorial and geometric constructions we make for $X$ are then derived from this new information. 

\subsection{The compactification $X_\sigma$}

We define a projective variety $\overline{M}(\sigma)$ using the same recipe used to define $M(\sigma)$:

\begin{equation}
\overline{M}(\sigma) = \prod_{e \in E^{\circ}(\sigma)} \overline{\SL}_2 \times \prod_{\ell \in L(\sigma)} \P^2.\\
\end{equation}

\noindent
The $\SL_2\times \SL_2$ and $\SL_2\times\G_m$ actions on $SL_2$ and $\A^2$ respectively both extend to their compactifications $\overline{\SL}_2$ and $\P^2$.  It follows that there is an action of $G(\sigma)$ on $\overline{M}(\sigma)$.  The line bundles defined in Proposition \ref{sl2geometry} on $\overline{\SL}_2$ and $\P^2$ (both denoted $\O(1)$ by abuse of notation) are linearized with respect to the actions on these spaces; it follows that the outer tensor product bundle $\L = \boxtimes_{e \in E(\tree)} \O(1)$ is $G(\sigma)$-linearized as well.  With these observations in mind we define $X_\sigma$ as the corresponding GIT quotient:

\begin{equation}
X_\sigma = \overline{M}(\sigma) \qr_\L G(\sigma).\\
\end{equation}

Before we show that $X_\sigma$ is indeed a compactification of $X$ (see Proposition \ref{compact}), we give a more detailed description of the coordinate ring $\C[X] = \C[M(\sigma)]^{G(\sigma)}$ and the projective coordinate ring $\C[X_\sigma] = \bigoplus_{n \geq 0} H^0(\overline{M}(\sigma), \L^{\otimes n})^{G(\sigma)}$ in terms of the tree $\sigma$.  In the sequel we will refer to a $\sigma$-weight $\s \in \Z_{\geq 0}^{E(\sigma)}$, which is an assignment of non-negative integers to the edges of $\sigma$.   The following decompositions of the coordinate ring of $\C[M(\sigma)]$ and the projective coordinate ring $\C[\overline{M}(\sigma)] = \bigoplus_{n \geq 0} H^0(\overline{M}(\sigma), \L^{\otimes n})$ can be computed from the isotypical decompositions of $\C[\SL_2]$, $\C[\A^2]$, $\C[\P^2] = \bigoplus_{n \geq 0} H^0(\P^2, \O(n))$ and $\C[\overline{\SL}_2] = \bigoplus_{n \geq 0} H^0(\overline{\SL}_2, \O(n))$:\\

\begin{equation}\label{invariantdecomposition}
\C[X] = \C[M(\sigma)]^{G(\sigma)} = \big[\bigotimes_{e \in E^{\circ}(\sigma)} \C[\SL_2]  \otimes \bigotimes_{\ell \in L(\sigma)} \C[\A^2] \big]^{G(\sigma)} =\\ 
\end{equation}
$$\bigoplus_{\s \in \Z_{\geq 0}^{E(\sigma)}} \big[ \bigotimes_{e \in E^{\circ}(\sigma)} V(\s(e))\otimes V(\s(e)) \otimes \bigotimes_{\ell \in L(\sigma)} V(\s(\ell)) \big]^{G(\sigma)}.$$\\

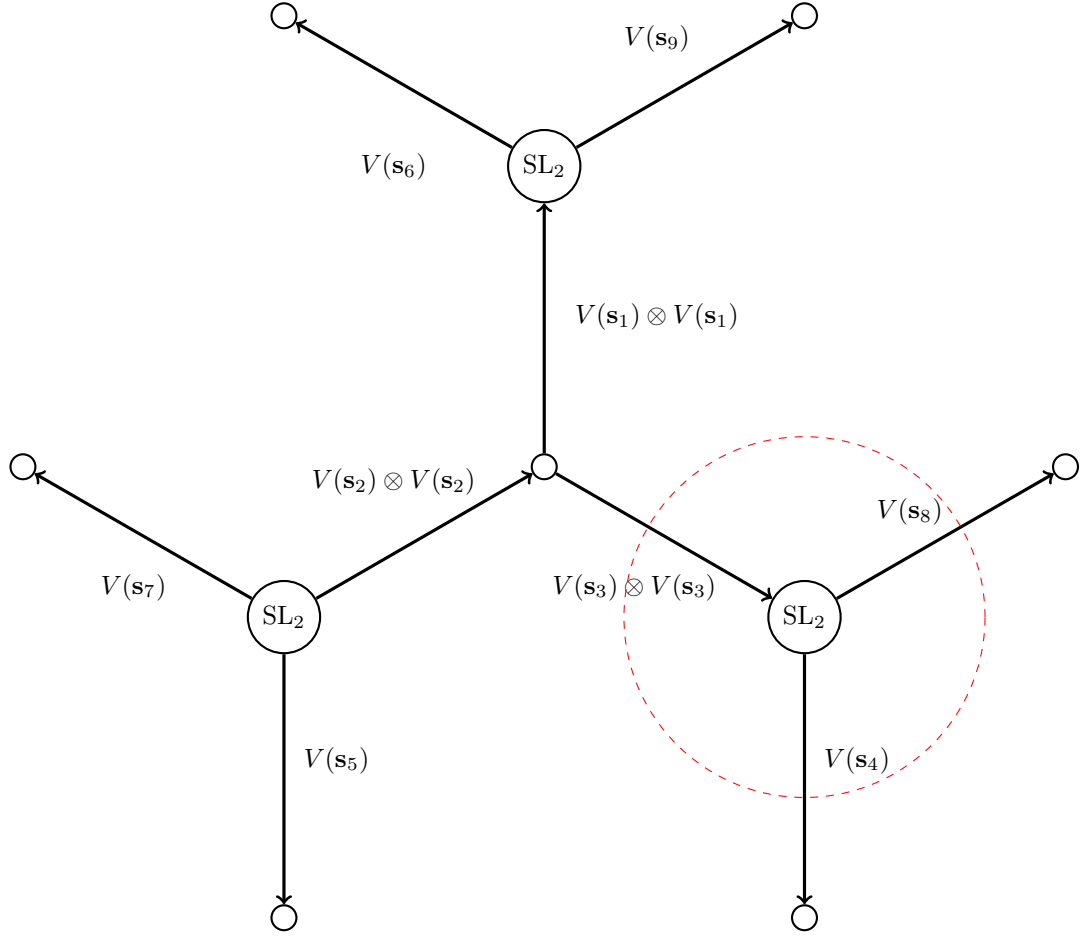
\begin{figure}
\begin{tikzpicture}[
every edge/.style = {draw=black,very thick, ->},
 vrtx/.style args = {#1/#2}{%
      circle, draw, thick, fill=white,
      minimum size=5/2mm, label=#1:#2}]
\node(A) [vrtx=left/] at (0, 0) {};
\node(B) [vrtx=left/] at (0,4) {$\SL_2$};
\node(C) [vrtx=left/] at (-3.46,-2) {$\SL_2$};
\node(D) [vrtx=left/] at (3.46, -2) {$\SL_2$};
\node(E) [vrtx=left/] at (-3.46,6){};
\node(F) [vrtx=left/] at (3.46,6){};
\node(G) [vrtx=left/] at (-6.93,0){};
\node(H) [vrtx=left/] at (-3.46,-6){};
\node(I) [vrtx=left/] at (6.93,0){};
\node(J) [vrtx=left/] at (3.46,-6){};
\node(L) at (1.5, 2){$V(\s_1)\otimes V(\s_1)$};
\node() at (-2, -.2){$V(\s_2)\otimes V(\s_2)$};
\node() at (1.2, -1.6){$V(\s_3)\otimes V(\s_3)$};
\node() at (3.46+.7, -3.9){$V(\s_4)$};
\node() at (-3.46+.7, -3.9){$V(\s_5)$};
\node() at (-2,4){$V(\s_6)$};
\node() at ((-3.46 -2, -1.6){$V(\s_7)$};
\node() at ((3.46 +1.4,-.6){$V(\s_8)$};
\node() at ((1.5, 5.7){$V(\s_9)$};
\path   (A) edge (B)
           (C) edge (A)
	(A) edge (D)
	(B) edge (E)
	(B) edge (F)
	(C) edge (G)
	(C) edge (H)
	(D) edge (I)
	(D) edge (J);
\draw [red, dashed] (3.46, -2) circle [radius = 2.4];
\end{tikzpicture}
\caption{Isotypical components $W_\sigma(\s) \subset \C[M(\sigma)]^{G(\sigma)}$. The dotted circle contains those $\SL_2$ representations which are acted on by the copy of $\SL_2$ associated to the lower right trinode. }
\end{figure}

\noindent
To ease notation we let $W_{\sigma}(\s) = \big[ \bigotimes_{e \in E^{\circ}(\sigma)} V(\s(e))\otimes V(\s(e)) \otimes \bigotimes_{\ell \in L(\sigma)} V(\s(\ell)) \big]^{G(\sigma)}$, so that $\C[X] = \bigoplus_{\s \in \Z_{\geq 0}^{E(\sigma)}} W_\sigma(\s)$.  Roughly speaking, each $\s \in \Z_{\geq}^{E(\sigma)}$ assigns two irreducible representations $V(\s(e))\otimes V(\s(e))$ to each non-leaf edge $e \in E^{\circ}(\sigma)$, one for the head of $e$ and one for the tail of $e$.  This pair is acted on through the right and left actions of $\SL_2$ on the copy of $\SL_2$ assigned to $e$.  Similarly, $\s$ assigns one representation $V(\s(\ell))$ to each leaf-edge $\ell \in L(\sigma)$; this space is acted on by $\SL_2$ through the left action on $\A^2$. Note that for any $\s \in \Z_{\geq 0}^{E(\sigma)}$, the space $W_\sigma(\s)$ can be written as the following tensor product:

\begin{equation}\label{vertex decomposition}
W_\sigma(\s) = \bigotimes_{v \in V(\sigma)} \big[ V(\s(e_1(v))) \otimes \cdots \otimes V(\s(e_k(v)))\big]^{\SL_2}.
\end{equation}

\noindent
where $e_1(v), \ldots, e_k(v)$ are the edges of $\sigma$ containing $v$.   The following lemma will be useful in Sections \ref{coneofvaluations} and \ref{tropical}.

\begin{lemma}\label{furtherdecompose}
Let $\sigma'$ be a tree with $n$ leaves which is obtained from $\sigma$ by contracting an edge $e \in E(\sigma)$, then there is a corresponding direct sum decomposition:

\begin{equation}
W_{\sigma'}(\s') = \bigoplus_{\{\s \in \Z_{\geq 0}^{E(\sigma)} \ \mid \ \forall e'\in E(\sigma'), \ \s(e') = \s'(e') \}} W_\sigma(\s).\\
\end{equation}

\end{lemma}

\begin{proof}
Let $v_1, v_2$ be the endpoints of $e$, and let $v\in V(\sigma')$ be the vertex created by bringing $v_1$ and $v_2$ together. 
We prove this lemma by considering the link of $v$ in $\sigma'$.  Everything we do is compatible with the geometric arguments given in Proposition \ref{GITXspace}. Let $e_1, \ldots, e_k$ be the edges of $\sigma'$ which contain $v$, and let $e_1, \ldots, e_s, e$ and $e, e_{s+1}, \ldots, e_k$ be the edges of $\sigma$ which contain $v_1$ and $v_2$, respectively. For the sake of simplicity, we orient all edges $e_i$ away from the vertices, and we have $e$ point from $v_1$ to $v_2$.   Pick $\a = (a_1, \ldots, a_k) \in \Z_{\geq 0}^k$ and $n \geq 0$, and consider the isotypical component of $ \C[\SL_2^s\times \SL_2 \times \SL_2^{k-s}]^{\SL_2 \times \SL_2}$:

\begin{equation}
\big[V(a_1) \otimes \cdots \otimes V(a_s)\otimes  V(n)\big]^{\SL_2} \otimes \big[V(n)\otimes V(a_{s+1})\otimes \cdots  \otimes V(a_k)\big]^{\SL_2} \otimes \big[V(a_1) \otimes \cdots \otimes V(a_k)\big].\\
\end{equation}

\noindent
The map $\SL_2^k \to \SL_2^s\times \SL_2 \times \SL_2^{k-s}$ which sends $(g_1, \ldots, g_k)$ to $(g_1, \ldots, g_s, Id, g_{s+1}, \ldots, g_k)$
induces the isomorphism of $\SL_2^k$-algebras $\C[\SL_2^k]^{\SL_2} \cong \C[\SL_2^s\times \SL_2 \times \SL_2^{k-s}]^{\SL_2 \times \SL_2}$ from Proposition \ref{GITXspace}.  This algebra map is computed on the above component by plugging the $V(n)$ component into its ``dual" $V(n)$. Since this map preserves the $\SL_2^k$ action, it must likewise map the $\a$ component of $\C[\SL_2^s\times \SL_2 \times \SL_2^{k-s}]^{\SL_2 \times \SL_2}$ isomorphically onto the $\a$-component of $\C[\SL_2^k]^{\SL_2}$, so we obtain  $\big[V(a_1)\otimes \cdots \otimes V(a_k)\big]^{\SL_2} \otimes V(a_1)\otimes \cdots \otimes V(a_k)$ as a direct sum over $n$ of the components above.  
\end{proof}


We make use of the decompositions of the Rees algebras $R $ and $S$ from Section \ref{sl2} to give a description of $\C[X_\sigma]$ in terms of the spaces $W_\sigma(\s)$.  By definition we have:

\begin{equation}
H^0(\overline{M}(\sigma), \L^{\otimes n}) = \bigotimes_{e \in E^{\circ}(\sigma)} H^0(\overline{SL}_2, \O(n)) \otimes \bigotimes_{\ell \in L(\sigma)} H^0(\P^2, \O(n)).
\end{equation}

\noindent
In particular, the same power $n$ is used in the computations of the global sections for each line bundle.  Recall that $H^0(\overline{\SL}_2, \O(n)) = \bigoplus_{0 \leq m \leq n} V(m) \otimes V(m) t^n$ and $H^0(\P^2, \O(n)) = \bigoplus_{0 \leq m \leq n} V(m)t^n$.  Since $t^n$ is a placeholder which agrees across all components of the tensor product, we obtain:

\begin{equation}
H^0(\overline{M}(\sigma), \L^{\otimes n}) =  \bigoplus_{\{\s \in \Z_{\geq 0}^{E(\sigma)} \ \mid \ \forall e\in E(\sigma), \ \s(e) \leq n\}} \big[ \bigotimes_{e \in E^{\circ}(\sigma)} V(\s(e))\otimes V(\s(e)) \otimes \bigotimes_{\ell \in L(\sigma)} V(\s(\ell))\big].
\end{equation}

As a consequence we obtain the following decomposition of the projective coordinate ring of $X_\sigma$:

\begin{equation}
\C[X_\sigma] = \C[\overline{M}(\sigma)]^{G(\sigma)} =  \bigoplus_{n\ge 0}\bigoplus_{\{\s \in \Z_{\geq 0}^{E(\sigma)} \ \mid \ \forall e\in E(\sigma),\ \s(e) \leq n\}} W_\sigma(\s)t^n.
\end{equation}

\begin{proposition}\label{compact}
The projective variety $X_\sigma$ is a compactification of $X$. 
\end{proposition}

\begin{proof}
We show that $X_\sigma = Proj(\C[X_\sigma])$ contains $X$ as a dense, open subset.   We consider the element $1t \in W_\sigma({\bf 0})t^1 \subset \C[X_\sigma]$, where ${\bf 0}: E(\sigma) \to \Z_{\geq 0}$ is the weight which assigns $0$ to every edge of $\sigma$.  As constructed, each graded component $\bigoplus_{\{\s \in \Z_{\geq 0}^{E(\sigma)} \ \mid \ \s(e) \leq n, \ \forall e \in E(\sigma)\}} W_\sigma(\s)t^n \subset \C[X_\sigma]$ is a subspace of $\C[X]$, and the multiplication operation on these graded components is computed by the multiplication rule in $\C[X]$; this is a consequence of the Proposition \ref{GITXspace} and the definition of $\C[X_\sigma]$. By inverting $1t$ we obtain $\frac{1}{1t^n}W_\sigma(\s)t^n = \frac{1}{1t^m}W_{\sigma}(\s)t^m$ for all $\s$ with $\s(e) \leq n, m$, $\forall e \in E(\sigma)$ in the $0$-degree part of $\frac{1}{1t}\C[X_\sigma]$.  It follows that $\big[\frac{1}{1t}\C[X_\sigma]\big]_0 \cong \C[X]$, and that the complement of the hypersurface $1t = 0$ in $X_\sigma$ is $Spec(\C[X]) = X$. 
\end{proof}

In the sequel, we let $D_\sigma \subset X_\sigma$ be the zero set of $1t \in \C[X_\sigma]$.

\section{The cone $C_\sigma$ of valuations on $\C[X]$}\label{coneofvaluations}

Our goal is to describe the geometry of the hypersurface $D_\sigma \subset X_\sigma$.  In order to construct its irreducible components and describe their intersections, we construct a cone $C_\sigma$ of discrete valuations on $\C[X]$. We show that $C_\sigma$ is simplicial and generated by distinguished valuations $v_e$, $e \in E(\sigma)$.  In Section \ref{divisorstructuresection} we show that $v_e$ is obtained by taking order of vanishing along a component of $D_\sigma$.

\subsection{Valuations on $\C[M(\sigma)]$}

We introduce a valuation $v_e: \C[X] \setminus \{0\} \to \Z$ for each edge $e \in E(\sigma)$.  First, we recall the valuations $v: \C[\SL_2] \setminus \{0\} \to \Z$ and $deg: \C[\A^2] \setminus\{0\} \to \Z$ from Section \ref{sl2}.  The space $M(\sigma)$ is the product $\prod_{e \in E^{\circ}(\sigma)} \SL_2 \times \prod_{\ell \in L(\sigma)} \A^2$. Accordingly, its coordinate ring carries a valuation $\bar{v}_e: \C[M(\sigma)]\setminus\{0\} \to \Z$ for each edge $e \in E(\sigma)$; this is computed by using $v$ when $e \in E^{\circ}(\sigma)$ and $deg$ when $e \in L(\sigma)$.  The associated algebraic filtration by the spaces $\bar{F}^e_m = \{f \in \C[M(\sigma)] \mid \bar{v}_e(f) \geq - m\}$ are given by the following spaces:

\begin{equation}
\bar{v}_e, e \in E^{\circ}(\sigma): \ \ \ \ \bar{F}^e_m = \big[ \bigotimes_{e' \in E^{\circ}(\sigma), \ e' \neq e}\C[\SL_2] \big] \otimes \big[ \bigoplus_{0 \leq n \leq m} V(n) \otimes V(n)\big] \otimes \big[ \bigotimes_{\ell \in L(\sigma)} \C[\A^2]\big],\\
\end{equation}

\begin{equation}
\bar{v}_\ell, \ell \in L(\sigma): \ \ \ \ \bar{F}^\ell_m = \big[ \bigotimes_{e \in E^{\circ}(\sigma)}\C[\SL_2] \big] \otimes \big[\bigoplus_{0 \leq n \leq m} V(n) \big]\otimes \big[ \bigotimes_{\ell' \in L(\sigma), \ \ell' \neq \ell} \C[\A^2]\big].\\
\end{equation}

\noindent
We also have the following strict filtration spaces:

\begin{equation}
\bar{v}_e, e \in E^{\circ}(\sigma): \ \ \ \ \bar{F}^e_{<m} = \big[ \bigotimes_{e' \in E^{\circ}(\sigma), \ e' \neq e}\C[\SL_2] \big] \otimes \big[ \bigoplus_{0 \leq n < m} V(n) \otimes V(n)\big] \otimes \big[ \bigotimes_{\ell \in L(\sigma)} \C[\A^2]\big],\\
\end{equation}

\begin{equation}
\bar{v}_\ell, \ell \in L(\sigma): \ \ \ \ \bar{F}^\ell_{<m} = \big[ \bigotimes_{e \in E^{\circ}(\sigma)}\C[\SL_2] \big] \otimes \big[\bigoplus_{0 \leq n < m} V(n) \big]\otimes \big[ \bigotimes_{\ell' \in L(\sigma), \ \ell' \neq \ell} \C[\A^2]\big]. \\
\end{equation}

\noindent
Clearly $\bar{F}^e_{< m} \subset \bar{F}^e_m$ for any $e \in E(\sigma)$. 
The reader can verify that the associated graded algebra of $\bar{v}_e, e \in E^{\circ}(\sigma)$ and $\bar{v}_{\ell}, \ell \in L(\sigma)$ are the coordinate rings of $\prod_{e' \in E^{\circ}(\sigma), \ e' \neq e} \SL_2 \times \SL_2^c \times \prod_{\ell \in L(\sigma)} \A^2$ and $\prod_{e \in E^{\circ}(\sigma)} \SL_2 \times \A^2 \times \prod_{\ell' \in L(\sigma), \ \ell' \neq \ell} \A^2$, respectively.  We observe that, for any $r \in \R_{\geq 0}$, the function $rv: \C[M(\sigma)]\setminus \{0\}\to \R$ is also a valuation. Furthermore, since $deg$ is obtained from a $\Z$-grading on $\C[\A^2]$, for any $r \in \R$, $rdeg: \C[\A^2] \setminus\{0\} \to \R$ is a valuation.  

\begin{definition}
Let $\r \in \R_{\geq 0}^{E^{\circ}(\sigma)}\times \R^{L(\sigma)}$, and let $\bar{v}_\r: \C[M(\sigma)] \setminus \{0\} \to \R$ be the valuation $\sum \r(e)\bar{v}_e$ obtained using the sum operation described in Definition \ref{valsum}. 
\end{definition}

The valuations $\bar{v}_\r$ are built from the valuations $v$ and $deg$, which can be computed entirely in terms of the representation theory of $\SL_2$.  The following lemma shows that this is also the case for $\bar{v}_\r$.

\begin{lemma}\label{upstairsevaluation}
Let $\s \in \Z_{\geq 0}^{E(\sigma)}$ and $f \in \big[ \bigotimes_{e \in E^{\circ}(\sigma)} V(\s(e))\otimes V(\s(e)) \otimes \bigotimes_{\ell \in L(\sigma)} V(\s(\ell))\big] \subset \C[M(\sigma)]$, then $\bar{v}_\r(f)$ is computed by taking the ``dot product" of $\r$ and $\s$ over the edges of $\sigma$:

\begin{equation}
\bar{v}_\r(f) = \sum_{e \in E(\sigma)} -\r(e)\s(e) = -\langle \r, \s \rangle.\\
\end{equation}

\noindent
Furthermore, the filtration space $\bar{F}^\r_m = \{ f \in \C[M(\sigma)] \mid \bar{v}_\r(f) \geq - m\}$ is the following sum:

\begin{equation}
\bar{F}^\r_m = \bigoplus_{\{\s \mid -\langle \r, \s \rangle \geq -m\}} \big[ \bigotimes_{e \in E^{\circ}(\sigma)} V(\s(e))\otimes V(\s(e)) \otimes \bigotimes_{\ell \in L(\sigma)} V(\s(\ell))\big].
\end{equation} 
\end{lemma}

\begin{proof}
This is a direct consequence of the formula for computing the valuations $v: \C[\SL_2]\setminus\{0\} \to \Z$ and $deg: \C[\A^2]\setminus\{0\} \to \Z$, and Definition \ref{valsum}. 
\end{proof}

\subsection{Valuations on $\C[X]$}\label{valuationsonX}

In what follows we place a partial ordering $\preceq$ on $\Z_{\geq 0}^{E(\sigma)}$, where $\s \preceq \s'$ if $\s(e) \leq \s'(e), \ \forall e \in E(\sigma)$.  We let $v_\r: \C[X] \setminus \{0\} \to \R$ be the restriction of $\bar{v}_\r$ from $\C[M(\sigma)]$ to $\C[X]$.  

\begin{proposition}\label{multrule}
The following hold for a tree $\sigma$, the associated decomposition $\C[X] = \bigoplus_{\s \in \Z_{\geq 0}^E(\sigma)} W_\sigma(\s)$, and the valuation $v_\r$ for any $\r \in \R_{\geq 0}^{E^{\circ}(\sigma)}\times \R^{L(\sigma)}$:\\

\begin{enumerate}
\item for $f \in W_\sigma(\s)$,  we have $v_\r(f) = -\langle \r, \s \rangle$,\\
\item for $m \in \R$, the filtration space $F^\r_m = \{ f \in \C[X] \mid v_\r(f) \geq - m\} = \bigoplus_{\{\s \mid -\langle \r, \s \rangle \geq -m\}} W_\sigma(\s)$,\\
\item for $f = \sum f_\s$ with $f_\s \in W_\sigma(\s)$, we have $v_\r(f) = MIN\{ -\langle \r, \s \rangle \mid f_\s \neq 0\}$,\\
\item for any $\s, \s' \in \Z_{\geq 0}^{E(\sigma)}$, we have $W_\sigma(\s)W_\sigma(\s') \subset \bigoplus_{\s'' \preceq \ \s + \s'} W_\sigma(\s'')$. Furthermore, the $\s + \s'$ component of this product is always nonzero.\\
\end{enumerate}

\end{proposition}

\begin{proof}
The valuations $\bar{v}_\r$ are all $G(\sigma)$-invariant and,  in particular,  their filtration spaces $\bar{F}^\r_m$ are $G(\sigma)$-representations.  We have $W_\sigma(\s) = \big[ \bigotimes_{e \in E^{\circ}(\sigma)} V(\s(e))\otimes V(\s(e)) \otimes \bigotimes_{\ell \in L(\sigma)} V(\s(\ell))\big]^{G(\sigma)}$, so $(1)$ is a consequence of Lemma  \ref{upstairsevaluation}.  Furthermore, to prove $(2)$ we can compute $F^\r_m = \bar{F}^\r_m \cap \C[X]$ $= \big[\bar{F}^\r_m\big]^{G(\sigma)}$ $= \bigoplus_{\{\s \mid \langle \r, \s \rangle \geq -m\}} W_\sigma(\s)$ by Lemma \ref{upstairsevaluation}.  Part $(2)$ shows that $v_\r$ is adapted to the direct sum decomposition $\C[X] = \bigoplus_{\s \in \Z_{\geq 0}^{E(\sigma)}} W_\sigma(\s)$ (recall this notion from Section \ref{valuations}), so part $(3)$ follows as a consequence. 

We know that $W_\sigma(\s)W_\sigma(\s') \subset \bigoplus_{\s'' \preceq \ \s + \s'} W_\sigma(\s'')$ from properties of multiplication in $\C[\SL_2]$ and $\C[\A^2]$.  For the second part of $(4)$, we first observe that $F^\r_{<m} = \bigoplus_{\{\s \mid \langle \r, \s \rangle > -m\}} W_\sigma(\s)$.  Let $v_\sigma$ be the valuation obtained from $\bar{v}_\sigma = \sum_{e \in E(\sigma)} \bar{v}_e$. Then, for any $f \in W_\sigma(\s'')$ with $\s'' \preceq \s + \s'$ we must have $v_\sigma(f) \geq \sum_{e \in E(\sigma)} \s(e) + \s'(e)$, with equality if and only if $\s'' = \s + \s'$.  Now $(2)$ implies that the product of the components $W_\sigma(\s)$ and $W_\sigma(\s')$ in the associated graded algebra $gr_\sigma(\C[X])$ of $v_\sigma$ is projection onto the $W_\sigma(\s + \s')$ component.  Since $v_\sigma$ is a valuation, $gr_\sigma(\C[X])$ is a domain, so this product must be nonzero. 
\end{proof}

Recall the Berkovich analytification $X^{an}$ of the affine variety $X$.  Proposition \ref{multrule} allows us to construct a distinguished subset of $X^{an}$ associated to the tree $\sigma$. 

\begin{corollary}
There is a continuous map $\phi_\sigma: \R_{\geq 0}^{E^{\circ}(\sigma)}\times \R^{L(\sigma)} \to X^{an}$ which takes $\r$ to $v_\r$.  
\end{corollary}

\begin{proof} 
In Proposition \ref{multrule}, we have shown that there is such a map $\phi_{\sigma}$. Thus, it remains to establish that this map is continuous.  Using the definition of the topology on $X^{an}$, it suffices to show that any evaluation function $ev_f$, $f \in \C[X]$, pulls back to a continuous function on $\R_{\geq 0}^{E^{\circ}(\sigma)}\times \R^{L(\sigma)}$.  By part $(3)$ of Proposition \ref{multrule}, we have $ev_f(v_\r) = MIN\{ -\langle \r, \s \rangle \mid f_\s \neq 0\}$, where $f_\s$ denotes the $W_\sigma(\s)$ component of $f$; this function is piecewise-linear in $\r$ and therefore continuous.
\end{proof}

Suppose that a tree $\sigma'$ is obtained from $\sigma$ by contracting a non-leaf edge $e \in E^{\circ}(\sigma)$.  There is a natural inclusion, $i_e: \R_{\geq 0}^{E^{\circ}(\sigma')}\times \R^{L(\sigma')} \to \R_{\geq 0}^{E^{\circ}(\sigma)}\times \R^{L(\sigma)}$,  by regarding $\R_{\geq 0}^{E^{\circ}(\sigma')}\times \R^{L(\sigma')}$ as the weightings of $\sigma$ which are $0$ on $e$.  

\begin{lemma}\label{faceinclude}
For $\r \in \R_{\geq 0}^{E^{\circ}(\sigma')}\times \R^{L(\sigma')}$, $v_\r = v_{i_e(\r)}$. As a consequence, $\R_{\geq 0}^{E^{\circ}(\sigma')}\times \R^{L(\sigma')}$ can be regarded as a face of $\R_{\geq 0}^{E^{\circ}(\sigma)}\times \R^{L(\sigma)}$. 
\end{lemma}

\begin{proof}
This follows by direct computation using Lemma \ref{furtherdecompose} and Proposition \ref{multrule}. 
\end{proof}

\begin{definition}\label{treespace}
Let $\mathcal{T}(n)$ denote the complex $\bigcup_{\sigma} \R_{\geq 0}^{E^{\circ}(\sigma)} \times \R^{L(\sigma)}$ obtained as the push-out of the diagram of inclusions defined by the maps $i_e$.
\end{definition}

The maps $\phi_\sigma$ glue together to define a continuous map $\Phi: \mathcal{T}(n) \to X^{an}$.  Let $C_\sigma$ denote the image, $\phi_\sigma(\R_{\geq 0}^{E^{\circ}(\sigma)}\times \R^{L(\sigma)})$.  In Section \ref{tropical} we show that the evaluation functions $ev_{p_{ij}}\circ \Phi: \mathcal{T}(n) \to \R$ map a tree to its \emph{dissimilarity vector} in $\R^{\binom{n}{2}}$; as a consequence, $\Phi$ must be an injective map.

\section{The geometry of $D_\sigma \subset X_\sigma$}\label{divisorstructuresection}

With the valuations $v_\r$ and the decomposition $\C[X] = \bigoplus_{\s \in \Z_{\geq 0}^{E(\sigma)}} W_\sigma(\s)$, we have two useful tools for understanding the geometry of the compactification $X_\sigma$.   In this section we show that $D_\sigma$ is reduced, and we give a recipe to decompose $D_\sigma$ into irreducible components.

\subsection{The ideal $I_S \subset \C[X_\sigma]$}

Much of our understanding of the divisor $D_\sigma$ is derived from the decomposition of the projective coordinate ring of $X_\sigma$ into the spaces $W_\sigma(\s)$:

\begin{equation}
\C[X_\sigma] = \bigoplus_{n\ge 0}\bigoplus_{\{\s \in \Z_{\geq 0}^{E(\sigma)} \mid \ \forall e \in E(\sigma),\ \s(e) \leq n\}} W_\sigma(\s) t^n.\\
\end{equation}

\noindent
This decomposition enables us to define a set of distinguished ideals in $\C[X_\sigma]$. 

\begin{definition}
For $S \subset E(\sigma)$,  a subset of the edges of $\sigma$,  let $I_S \subset \C[X_\sigma]$ be the following vector space:\\

\begin{equation}
I_S = \bigoplus_{n\ge 0}\bigoplus_{\{\s \in \Z_{\geq 0}^{E(\sigma)} \mid \ \forall e \in E(\sigma),\ \s(e) \leq n,\  \exists e' \in S, \ \s(e') < n\}} W_\sigma(\s) t^n.\\
\end{equation}

\end{definition}

\begin{proposition}\label{ideal}
For any $S \subset E(\sigma)$, $I_S \subset \C[X_\sigma]$ is a prime ideal.
\end{proposition}

\begin{proof}
We recall the multiplication rule, $W_\sigma(\s)W_\sigma(\s') \subset \bigoplus_{\s'' \preceq \ \s + \s'} W_\sigma(\s'')$, from Proposition \ref{multrule}, and immediately deduce that $I_S$ is a homogeneous ideal.  Now suppose $fg \in I_S$ for $f, g \in \C[X_\sigma]$, homogeneous elements of degrees $n$ and $m$, respectively. We view $f$ and $g$ as regular functions on $X$ with the property that $v_e(f) \leq n$  and $v_e(g) \leq m$, $\forall e \in E(\sigma)$.  If $fg \in I_S$, it must be the case that $v_e(fg) = v_e(f) + v_e(g) < n + m$ for some $e \in E(\sigma)$. But, this can only happen if $v_e(f) < n$ or $v_e(g) < m$, so we conclude that $f \in I_S$ or $g \in I_S$.  
\end{proof}

We let $D_S \subset X_\sigma$ be the zero locus of $I_S$.  Clearly we have that $S \subset S'$ implies that $I_S \subset I_{S'}$ and $D_S \supset D_{S'}$.  The following proposition shows that $D_\sigma$ is built from the irreducible, reduced subvarieties $D_S$. 

\begin{proposition}\label{divisorstructure}
For any tree $\sigma$ the following hold:\\

\begin{enumerate}
\item $I_{S \cup S'} = I_S + I_{S'}$, $D_S \cap D_{S'} = D_{S \cup S'}$,\\
\item if $S \neq S'$ then $I_S \neq I_{S'}$.
\end{enumerate}

\end{proposition}

\begin{proof}
All of the ideals $I_S$ are sums of the spaces $W_\sigma(\s)t^n$, so for both $(1)$ and $(2)$ it suffices to check membership on these spaces. We have $W_\sigma(\s)t^n \subset I_{S \cup S'}$ if and only if there is some $e \in S \cup S'$ with $\s(e) < n$; this happens if and only if $W_\sigma(\s)t^n$ is either in $I_S$ or $I_{S'}$.  To show that $I_S, I_{S'}$ are distinct prime ideals we only need to show that there is some element which is in $I_S$, but not in $I_{S'}$.  We define a weighting $\omega_S$ of $E(\sigma)$ by non-negative integers so that $0 \neq W_\sigma(\omega_S)t^4 \subset \C[X_\sigma]$, $\omega_S(e) = 4$, $\forall e \in S$ and $\omega_S(e) = 2$, $\forall e \notin S$.  Clearly $W_\sigma(\omega_S)t^4 \subset I_{S'}$ but $W_\sigma(\omega_S)t^4 \not\subset I_S$.   This reduces the question to showing that $W_\sigma(\omega_S) \neq 0$, which is handled by the following lemma. 
\end{proof}

For the proof of the following lemma see \ref{vertex decomposition}.

\begin{lemma}\label{invariantspacenot0}
For any subset $S \subset E(\sigma)$, the invariant space $W_\sigma(\omega_S) \subset \C[X]$ is nonzero. 
\end{lemma}

\begin{proof}
It suffices to show that any tensor product $V(n_1)\otimes \cdots \otimes V(n_k)$ where $n_i \in \{2, 4\}$ and $k \geq 3$
contains an invariant.  If $k = 3$, the Pieri rule (Equation \ref{pieri}) shows that this is the case.  Suppose that this holds
up to $k-1$. The tensor product decomposition $V(n_1) \otimes V(n_2) = \bigoplus V(m)$ induces a decomposition of the $k-$fold tensor product:

\begin{equation}
V(n_1) \otimes \cdots \otimes V(n_k) = \bigoplus V(m) \otimes V(n_3) \otimes \cdots \otimes V(n_k).\\
\end{equation}

\noindent
Here, the sum is over all $m$ so that $n_1, n_2, m$ satisfy the Pieri rule. For whatever combination of $2$ and $4$ are given by $n_1, n_2$, we know we can have $m$ be $2$ or $4$ as necessary from the case $k = 3$.  But then $V(m) \otimes \cdots \otimes V(n_k)$ contains an invariant by the induction hypothesis, so $V(n_1) \otimes \cdots \otimes V(n_k)$ does as well. 
\end{proof}

\begin{corollary}
If $\sigma$ is a trivalent tree, then the $D_S$ are the intersections of the irreducible components of the reduced divisor $D_\sigma$. In particular, $D_\sigma$ is of combinatorial normal crossings type. 
\end{corollary}

\begin{proof}
If $\sigma$ is trivalent, then $|E(\sigma)| = dim(X_\sigma) = 2n-3$. Picking any ordering on the elements $e_i \in E(\sigma)$ we can form an increasing chain of distinct subsets $S_i = \{e_1, \ldots, e_i\}$.  From Proposition \ref{divisorstructure}, we know that the corresponding ideals $I_{S_i}$ form an increasing chain of distinct prime ideals.  It follows that $height(I_S) = codim(D_S) = |S|$.  In particular, $codim(D_e) = 1$ for any $e \in E(\sigma)$, and $D_S = \cap_{e \in S} D_e$.    Finally we observe that $\langle 1t \rangle \subset \C[X_\sigma]$ is $\cap_{e \in E(\sigma)} I_e$; it follows that $D_\sigma = \cup_{e \in E(\sigma)} D_e$. 
\end{proof}

\begin{remark}
It is possible to show that $D_\sigma$ is a combinatorial normal crossings divisor when $\sigma$ is not trivalent, but we omit the proof here for the sake of simplicity. 
\end{remark}

\subsection{$D_S$ as a GIT quotient}\label{DSgeom}

Now we describe each reduced, irreducible subvariety $D_S \subset X_\sigma$ as a GIT quotient in the style of Section \ref{GIT}.  We select a direction on each $e \in E(\sigma)$ and define the following product space associated to $S \subset E(\sigma)$:

\begin{equation}
\overline{M}(\sigma, S) = \prod_{e \in E^{\circ}(\sigma) \setminus S} \overline{\SL}_2 \times \prod_{e \in E^{\circ}(\sigma) \cap S} (\P^1 \times \P^1) \times \prod_{\ell \in L(\sigma) \setminus S} \P^2 \times \prod_{\ell \in L(\sigma) \cap S} \P^1.\\
\end{equation}

\noindent
The space $\overline{M}(\sigma, S)$ is naturally a closed, reduced, irreducible subspace of $\overline{M}(\sigma)$.  We let $\L_S$ be the restriction of the line bundle $\L$ to this subspace.  Following Section \ref{sl2}, we have:

\begin{equation}
\L_S = \big[\boxtimes_{e \in E^{\circ}(\sigma) \setminus S} \O(1)\big] \boxtimes \big[\boxtimes_{e \in E^{\circ}(\sigma) \cap S}\O(1)\boxtimes\O(1) \big] \boxtimes \big[\boxtimes_{\ell \in L(\sigma) \setminus S}\O(1)\big] \boxtimes \big[\boxtimes_{\ell \in L(\sigma) \cap S} \O(1)\big].\\
\end{equation}

\noindent
where:

\begin{equation}
H^0(\overline{\SL}_2, \O(m)) = F_m = \bigoplus_{0 \leq n \leq m} V(n) \otimes V(n),\\
\end{equation}

\begin{equation}
H^0(\P^1\times\P^1, \O(m)\boxtimes\O(m)) = V(m)\otimes V(m),\\
\end{equation}

\begin{equation}
H^0(\P^2, \O(m)) = \bigoplus_{0 \leq n \leq m} V(n),\\
\end{equation}

\begin{equation}
H^0(\P^1, \O(m)) = V(m),\\
\end{equation}

\noindent
as $\SL_2$-representations. 

\begin{proposition}\label{GITexpression}
The space $D_S$ is isomorphic to the GIT quotient $\overline{M}(\sigma, S)\qr_{\L_S} G(\sigma)$.  In particular, there is an isomorphism of graded algebras:

\begin{equation}
\bigoplus_{m \geq 0} H^0(\overline{M}(\sigma, S), \L_S^{\otimes m})^{G(\sigma)} \cong \C[X_{\sigma}]/I_S.\\
\end{equation}
\end{proposition}

\begin{proof}
Using the descriptions of the components of the graded coordinate rings of $\overline{M}(\sigma, S)$ and $\overline{M}(\sigma)$ above and in Section \ref{sl2}, we can form the following exact sequence. 

\begin{equation}\label{exactsequence}
0 \to \bigoplus_{m \geq 0} J_m \to \bigoplus_{m \geq 0} H^0(\overline{M}(\sigma), \L^{\otimes m}) \to \bigoplus_{m \geq 0} H^0(\overline{M}(\sigma, S), \L_S^{\otimes m}) \to 0,\\
\end{equation}

\noindent
where $J_m$ is the direct sum of components of the form 

\begin{equation}
\big[\bigotimes_{e \in E^{\circ}(\sigma) \setminus S} V(\s(e)) \otimes V(\s(e))\big]\otimes \big[\bigotimes_{e \in E^{\circ}(\sigma) \cap S} V(\s(e))\otimes V(\s(e)) \big] \otimes \big[\bigotimes_{\ell \in L(\sigma) \setminus S} V(\s(\ell))\big] \otimes \big[\bigotimes_{\ell \in L(\sigma) \cap S} V(\s(\ell))\big],\\
\end{equation}

\noindent
for $\s \in \Z_{\geq 0}^{E(\sigma)}$ with $\s(e) \leq m$ for all $e \in E(\sigma)$ and $\s(e) < m$ (or $\s(\ell) < m$) for some $e \in S$.  Taking $G(\sigma)$-invariants of (\ref{exactsequence}) produces the exact sequence:

\begin{equation}
0 \to I_S \to \C[X_\sigma] \to \bigoplus_{m \geq 0} H^0(\overline{M}(\sigma, S), \L_S^{\otimes m})^{G(\sigma)} \to 0,\\
\end{equation}

\noindent
which proves the proposition. 
\end{proof}

\begin{remark}
Proposition \ref{GITexpression} gives an interesting interpretation of the compactification $D_\sigma \subset X_\sigma \supset X$, where ``going to infinity" in the direction of an edge $e \in E(\sigma)$ has one passing from $\SL_2 \subset \overline{\SL}_2$ to $\P^1 \times \P^1$.  This latter space is a $\G_m$-quotient of the singular $2\times2$ matrices $\SL_2^c$.  So in the boundary $D_\sigma$ we see a variant of a quiver variety made from singular matrices, as opposed to $X\subset X_\sigma$ which is made with elements of $\SL_2$.
\end{remark}

\subsection{The valuations $ord_{D_e}$}

Now we relate the divisorial valuations associated to the irreducible components of $D_\sigma$ to the cone of valuations $C_\sigma$ constructed in Section \ref{coneofvaluations}.  In Proposition \ref{sl2geometry},  we see that taking order along the boundary divisor $D = \P^1\times \P^1 \subset \overline{\SL}_2$ produces the valuation $v: \C[\SL_2] \setminus \{0\} \to \Z$.  An identical statement holds for the boundary copy of $\P^1$ in $\P^2 \supset \A^2$ and the valuation $deg: \C[\A^2] \setminus \{0\} \to \Z$.  

We have a distinguished irreducible divisor $\overline{M}(\sigma, e) \subset \overline{M}(\sigma)$ for each edge $e \in E(\sigma)$ coming from the construction in \ref{DSgeom}.  As a product space, $\overline{M}(\sigma, e)$ is obtained by replacing the copy of $\overline{\SL}_2$ or $\P^2$ at the edge $e$ with its boundary divisor.   Since $M(\sigma)$ is a dense, open subspace of $\overline{M}(\sigma)$, both the coordinate ring $\C[M(\sigma)]$ and its ring of $G(\sigma)$-invariants $\C[X]$ inherit the valuation $ord_{\overline{M}(\sigma, e)}$. The divisor $D_e$ is then obtained from $\overline{M}(\sigma, e)$ as the GIT quotient.  The following proposition relates the valuations obtained from these divisors.

\begin{proposition}\label{valscoincide}
For any $e \in E(\sigma)$, the valuations $v_e, ord_{\overline{M}(\sigma, e)}$, and $ord_{D_e}$ coincide on $\C[X]$. 
\end{proposition}

\begin{proof}
By definition we have $\bar{v}_e = ord_{\overline{M}(\sigma, e)}$, where $\bar{v}_e: \C[M(\sigma)] \setminus \{0\} \to \Z$ is from Section \ref{coneofvaluations}; so it follows that $v_e = ord_{\overline{M}(\sigma, e)}$ on $\C[X]$.  Let $\bar{\eta}_e$ be the generic point of $\overline{M}(\sigma, e) \subset \overline{M}(\sigma)$, with local ring $\O_{\bar{\eta}_e}$ and maximal ideal $\langle \bar{t}_e \rangle \subset \O_{\bar{\eta}_e}$. 
Then, the local ring $\O_{\eta_e}$ at the generic point $\eta_e$ of $D_e$ is the ring of $G(\sigma)$-invariants in $\O_{\bar{\eta}_e}$, and furthermore $\bar{t}_e \in \O_{\eta_e}$.  Computing $v_e$ must then coincide with computing $ord_{D_e}$, as both valuations amount to measuring $\bar{t}_e$ degree. 
\end{proof}

\subsection{$X_\sigma$ is Fano}

We finish this section with the observation that the compactifications $X_\sigma$ are all of Fano type.  For simplicity we focus on the case when $\sigma$ is a trivalent tree. We use a result of Watanabe \cite{Watanabe} which allows a computation for the anti-canonical class of the $Proj$ of a positively graded Cohen-Macaulay algebra.

\begin{proposition}\label{Fano}
For $\sigma$, a trivalent tree, and $D_\sigma \subset X_\sigma$, the associated boundary divisor of the compactification, we have 

\begin{equation}
-K_{X_\sigma} = 3D_\sigma\\
\end{equation}

\noindent
In particular, since $D_\sigma$ is  ample, $X_\sigma$ is Fano. 
\end{proposition}

\begin{proof}
By Proposition \ref{valscoincide}, $\C[X_\sigma] = \bigoplus_{n \geq 0} H^0(X_\sigma, nD_\sigma)$, so it follows that $D_\sigma$ is ample.    The projective coordinate rings of $\overline{\SL}_2$ and $\mathbb{P}^2$ are both normal; as a consequence, algebra $\C[\overline{M}(\sigma)]$ is normal.  Since $X_\sigma$ is a GIT quotient of $\overline{M}(\sigma)$, we must have that $\C[X_\sigma]$ is normal as well.  

Now we use some material from Section \ref{assgradTn} to produce a toric degeneration of $\C[X_\sigma]$. The basis $\B_\sigma \subset \C[X]$ induces a basis in $\bar{\B}_\sigma \subset \C[X_\sigma]$.  The members of $\bar{\B}_\sigma$ are labeled by elements $(\s, m)$ in the semigroup  $\bar{S}_\sigma \subset S_\sigma \times \Z$.  Here $\s \in S_\sigma$ and for each edge $e \in E(\sigma)$ we have $\s(e) \leq m$.  It is straightforward to show that this semigroup is generated by the elements $(\omega_{ij}, 1)$ where $\omega_{ij}$ is as in Proposition \ref{finitesemigroup}. The proof of Proposition \ref{assgrad} can be applied to $\C[X_\sigma]$ to show that $\C[X_\sigma]$ has associated graded algebra $\C[\bar{S}_\sigma]$.

Next, we show that $\C[X_\sigma]$ is a Gorenstein algebra. The algebra $\C[\bar{S}_\sigma]$ is a flat degeneration of $\C[X_\sigma]$, so by the argument in \cite[Proposition 3.7]{Lawton-Manon}, it suffices to show that $\C[\bar{S}_\sigma]$ is Gorenstein.  The algebra $\C[\bar{S}_\sigma]$ is a normal affine semigroup algebra, so we use \cite[Corollary 6.3.8]{BH} to show that it is Gorenstein. We consider $(\omega, 3) \in \bar{S}_\sigma$, where $\omega(e) = 2$ for all $e \in E(\sigma)$. If $(\tau, m) \in \bar{S}_\sigma$ is in the relative interior of the semigroup, we must have $\tau(e) < m$ and for any three edges $e, f, g$ meeting at a vertex we need $\tau(e) < \tau(f) + \tau(g)$.  Given these inequalities, it is straightforward to check that in this case $(\tau, m) - (\omega, 3)$ is still in $\bar{S}_\sigma$.  This proves that $\C[\bar{S}_\sigma]$ and $\C[X_\sigma]$ are Gorenstein algebras. 

Finally, we apply \cite[Corollary 2.9]{Watanabe} to $X_\sigma$ and its projective coordinate ring $\C[X_\sigma]$.   Since $(\omega, 3)$ is degree $3$ in $\C[\bar{S}_\sigma]$, as a consequence, the $a$-invariant of $\C[\bar{S}_\sigma]$ is $-3$. This information can be recovered from the Hilbert function of $\C[\bar{S}_\sigma]$, which agrees with the Hilbert function of $\C[X_\sigma]$.  It follows that the $a$-invariant of $\C[X_\sigma]$ is $-3$ as well.  Furthermore, $D_\sigma$ is a multiplicity-free sum of irreducible divisors, so $K_{X_\sigma} + 3D_\sigma = 0$ in $CL(X_\sigma)$.  
\end{proof}

\section{The tropical geometry of $X$}\label{tropical}

We recall the tropical variety $\trop(I_{2, n})$ obtained from the homogeneous Pl\"ucker ideal $I_{2, n}$.  The ideal $I_{2, n}$ vanishes on the Pl\"ucker generators $p_{ij} \in \C[X], \ 1 \leq i < j \leq n$, and defines the Pl\"ucker embedding $\Gr_2(\C^n) \subset \P(\bigwedge^2(\C^n))$ of the Grassmannian of $2$-planes.  We show that the map $ev_n = (\ldots, ev_{p_{ij}}, \ldots): X^{an} \to \R^{\binom{n}{2}}$ defined by the Pl\"ucker generators maps $\mathcal{T}(n)$ isomorphically onto $\trop(I_{2, n})$. 

\subsection{The tropical Grassmannian}

The tropical Grassmannian variety $\trop(I_{2, n})$ was introduced by Speyer and Sturmfels in \cite{Speyer-Sturmfels}.  It is one of the best understood tropical varieties in part because the Pl\"ucker relations are known to be a tropical basis for $I_{2, n}$: for any $1 \leq i, j, k, \ell \leq n$ in cyclic order we have 

\begin{equation}
p_{ij}p_{k\ell} - p_{i\ell}p_{jk} + p_{ik}p_{j\ell} = 0.\\
\end{equation}

In particular, the tropical variety $\trop(I_{2, n})$ is then the set of tropical solutions $\d = (\ldots, d_{ij}, \ldots) \in \R^{\binom{n}{2}}$ of the following tropical polynomials:

\begin{equation}
MIN\{d_{ij} + d_{k\ell}, d_{i\ell} + d_{jk}, d_{ik} + d_{j\ell}\}.\\
\end{equation}

\noindent
Using a variant of \cite[Theorem 4.2]{Speyer-Sturmfels}, it is then possible to use a solution $\d \in \trop(I_{2, n})$ to reconstruct a unique tree $\sigma$ with $n$ labeled leaves along with a corresponding real weight vector $\r \in \R_{\geq 0}^{E^{\circ}(\sigma)}\times \R^{L(\sigma)}$ such that $d_{ij}$ is the sum of entries $\r(e)$ along edges $e$ in the unique path in $\sigma$ between the leaves $i$ and $j$.  We let $\d: \R_{\geq 0}^{E^{\circ}(\sigma)}\times \R^{L(\sigma)} \to \R^{\binom{n}{2}}$ be the function which takes a metric tree to the vector of pairwise distances between its leaves; $\d(\r)$ is called the dissimilarity vector of $\r$. (see \cite{Manon-Dissimilarity-SL}, \cite{Manon-Dissimilarity-GL}, \cite{Pachter-Speyer}.) 

Now we show that $\trop(I_{2, n})$ can be realized as the image of $\mathcal{T}(n)$ under the evaluation map defined by the Pl\"ucker generators.  We let $ev_n: X^{an} \to \R^{\binom{n}{2}}$ be the map which sends $v \in X^{an}$ to $(\ldots, v(p_{ij}), \ldots)$. Recall the continuous map $\Phi: \mathcal{T}(n) \to X^{an}$ defined in Section \ref{valuationsonX}.

\begin{proposition}\label{pluckerevaluate}
The composition $ev_n \circ \Phi: \mathcal{T}(n) \to \R^{\binom{n}{2}}$ is an isomorphism of the complex of  polyhedral cones onto $\trop(I_{2, n})$. In particular,  $ev_n(v_\r)$ is equal to the dissimilarity vector $\d(\r) \in \R^{\binom{n}{2}}$. 
\end{proposition}

\begin{proof}
A variant of this proposition appears in \cite{Manon-Dissimilarity-SL}. First, consider the decomposition of $\C[X]$ given by its characterization as the ring of $\SL_2$-invariants in $\C[\A^2 \times \ldots \times \A^2]$:\\

\begin{equation}
\C[X] = \bigoplus_{\a \in \Z_{\geq 0}^n} [V(a_1)\otimes \cdots \otimes V(a_n)]^{\SL_2}.\\ 
\end{equation}

\noindent
The invariant space with $a_k = 0$ except for $a_i = a_j = 1$ is $1-$dimensional (see Section \ref{sl2}), and it is spanned
by the Pl\"ucker generator $p_{ij}$.  

Now we fix a tree $\sigma$, and consider the following decomposition, which can be derived from (\ref{invariantdecomposition}):\\

\begin{equation}
[V(a_1)\otimes \cdots \otimes V(a_n)]^{\SL_2} = \bigoplus_{\{\s \in \Z_{\geq 0}^{E(\sigma)} \ \mid \ \s(\ell_i) = a_i\}} W_\sigma(\s).\\
\end{equation}
 
\noindent
For the invariant space containing $p_{ij}$, exactly one $\s$ in this decomposition can have $W_\sigma(\s) \neq 0$.  Using (\ref{vertex decomposition}), we observe that the $\s$ with this property satisfies $\s(e) = 1$ if $e$ is in the unique path from $i$ to $j$ and $\s(e) = 0$ otherwise.  Now, Proposition \ref{multrule} implies that $v_\r(p_{ij}) = \langle \r, \s \rangle$ for this $\s$, which is precisely the sum of the $\r(e)$ for $e$ in the unique path from $i$ to $j$.  The characterization of $\trop(I_{2, n})$ given in \cite{Speyer-Sturmfels} now implies that the image of $C_\sigma$ under the map $ev_n$ is precisely the vectors $\d \in \R^{\binom{n}{2}}$ coming from trees with topology and labeling given by $\sigma$.  Each map $\r \to v_\r(p_{ij})$ is linear, so $ev_n$ maps $C_\sigma$ linearly onto its image.  
\end{proof}

\subsection{Associated graded algebras from $\mathcal{T}(n)$}\label{assgradTn}

Now we will compute the associated graded algebras of the valuations $v_\r \in \Phi(\mathcal{T}(n))$.  Lemma \ref{faceinclude} allows us to regard any valuation $v_\r \in C_{\sigma'}$ as $v_{i_e(\r)} \in C_\sigma$, where $\sigma'$ is obtained from $\sigma$ by contracting the edge $e$.  Repeatedly using Lemma \ref{faceinclude} therefore allows us to consider only trivalent trees when we compute with the valuations $v_\r$.  For now we assume that $\sigma$ is trivalent.  We can see from (\ref{vertex decomposition}) that each space $W_\sigma(\s) \subset \C[X]$ in this case is a tensor product of invariant spaces of the form $[V(i)\otimes V(j) \otimes V(k)]^{\SL_2}$.  The Pieri rule (\ref{pieri}) then implies the following lemma. 

\begin{lemma}
For $\sigma$ a trivalent tree, $W_\sigma(\s)$ is multiplicity-free.  In particular, $W_\sigma(\s) = \C$ if for every vertex $v \in V(\sigma)$ with edges $e_1, e_2, e_3$, the triple $\s(e_1), \s(e_2), \s(e_3)$ satisfies the conditions of the Pieri rule, and $W_\sigma(\s) = 0$ otherwise. 
\end{lemma}

\begin{definition}
For this definition, see (\ref{pieri}). Let $L_\sigma \subset \Z^{E(\sigma)}$ be the sublattice of those points $\omega$ with the property that $\omega(e_1), \omega(e_2), \omega(e_3)$ satisfy the \emph{parity condition} whenever $e_1, e_2, e_3$ share a common vertex.   Let $P_\sigma \subset \R_{\geq 0}^{E(\sigma)}$ be the polyhedral cone of those points $\omega$ with the property that $\omega(e_1), \omega(e_2), \omega(e_3)$ satisfy the \emph{triangle inequalities} whenever $e_1, e_2, e_3$ share a common vertex.   Finally, let $S_\sigma$ be the saturated affine semigroup $P_\sigma \cap L_\sigma$. 
\end{definition}

The coordinate algebra $\C[X]$ can now be expressed as a direct sum of $1$-dimensional spaces $W_\sigma(\s)$:

\begin{equation}
\C[X] = \bigoplus_{\s \in S_\sigma} W_\sigma(\s).\\
\end{equation}

\noindent
Choose one non-zero vector $b_\s \in W_\sigma(\s)$ for each $\s \in S_\sigma$ so that $\C[X] = \bigoplus_{\s \in S_\sigma} \C b_\s$.  Proposition \ref{multrule}  implies that multiplication of basis members has a lower-triangular expansion: $b_\s b_{\s'} = \sum_{\s'' \prec \s + \s'} C_{\s, \s'}^{\s''} b_{\s''}$, where $\prec$ indicates that $\s''(e) \leq \s(e) + \s'(e)$ for every $e \in E(\sigma)$. We call $\B_\sigma = \{b_\s \mid \s \in S_\sigma\}$ a \emph{branching basis} of $\C[X]$ corresponding to $\sigma$. The following is immediate from Definition \ref{adapted} and part $(2)$ of Proposition \ref{multrule}.

\begin{proposition}\label{branchingbasis}
A branching basis $\B_\sigma \subset \C[X]$ is adapted to every valuation in $C_\sigma$. 
\end{proposition}

Now we describe the associated graded algebras of the $v_\r \in \mathcal{T}(n)$.  As usual, let the tree $\sigma$ be equipped with an orientation on its edges, and let $S \subset E^{\circ}(\sigma)$ be some subset of non-leaf edges.  We define two affine schemes attached to this data:\\

\begin{equation}
M(\sigma, S) = \prod_{e \in E^{\circ}(\sigma)\setminus S} \SL_2 \times \prod_{e \in S} \SL_2^c \times \prod_{\ell \in L(\sigma)}\A^2,\\
\end{equation}

\begin{equation}
X(S) = M(\sigma, S)\qr G(\sigma).\\
\end{equation}

\noindent
Here $G(\sigma)$ acts on $M(\sigma, S)$ by the same recipe used on $M(\sigma)$. 

\begin{proposition}\label{assgrad}
For $v_\r \in C_\sigma$, let $S \subset E(\sigma)$ be the set of edges for which $\r(e) \neq 0$; we have the following:\\

\begin{enumerate}
\item the product $b_\s b_{\s'}$ in $gr_\r(\C[X])$ is the subsum of $\sum_{\s'' \prec \s + \s'} C_{\s, \s'}^{\s''} b_{\s''}$
consisting of those terms $\s''$ where $\s''(e) = \s(e) + \s'(e)$ when $e \in S$ and $\s''(e) \leq \s(e) + \s'(e)$ when $e \notin S$.\\
\item $gr_\r(\C[X]) \cong \C[X(S)]$,\\
\item the Pl\"ucker generators $p_{ij}$ are a Khovanskii basis for any $v_\r$.
\end{enumerate}

\end{proposition}

\begin{proof}
First, we observe that Proposition \ref{multrule} implies that the equivalence classes of the basis members $b_\s \in \B_\sigma$ are still a basis of $gr_\r(\C[X])$.  Indeed, any component $F^\r_m/F^\r_{<m} \subset gr_\r(\C[X])$ is a quotient of the span of $\B_\sigma \cap F^\r_m$ by the span of $\B_\sigma \cap F^\r_{<m}$.  In particular, for any $f = \sum C_\s b_\s \in \C[X]$, the equivalence class $\bar{f} \in gr_\r(\C[X])$ is computed by taking the subsum of only those terms $C_\s b_\s$ for which $\langle \r, \s \rangle$ is minimal.  Part $(1)$ follows from this observation.  

For part $(2)$, note that $\C[X(S)]$ also has a decomposition into the spaces $W_\sigma(\s)$. Indeed, the coordinate rings of $\SL_2^c$ and $\SL_2$ have exactly the same isotypical decomposition, however their multiplication rules are different.  In particular, the dominant weight decomposition defines a grading on $\C[\SL_2^c]$.  This implies that for the components corresponding to $e \in S$, the only $W_\sigma(\s'')$ which contribute to the expansion of $W_\sigma(\s)W_\sigma(\s')$ are those with $\s''(e) = \s'(e) + \s(e)$; this proves part $(2)$. 

For part $(3)$, we select $\r'$ with $\r'(e) > 0$ for every $e \in E(\sigma)$.  In this case, the expansion of $b_\s b_{\s'} \in gr_{\r'}(\C[X])$ only has a $\s + \s'$ component.  Since $\C$ is algebraically closed, it follows that $gr_{\r'}(\C[X]) \cong \C[S_\sigma]$ (see \cite{Eisenbud-Sturmfels}). But this top component is always there when this multiplication is carried out in $gr_\r(\C[X])$; so it follows (see Lemma \ref{adaptedquasi}) that $gr_{\r'}(gr_\r(\C[X])) \cong \C[S_\sigma]$.  This means that a generating set of $\C[S_\sigma]$ can be lifted to a generating set of $gr_\r(\C[X])$.  The following lemma then implies that (the equivalence classes of) the Pl\"ucker generators generate any $gr_\r(\C[X])$. 
\end{proof}

\begin{lemma}\label{finitesemigroup}
The affine semigroup $S_\sigma$ is generated by the weightings $\omega_{ij}$ ($1 \leq i < j \leq n$) which assign $1$ to every edge on the unique path between leaf $i$ and leaf $j$ and $0$ elsewhere. 
\end{lemma}

\begin{proof}
This is a standard result, see e.g. \cite[Proposition 4.6]{Howard-Manon-Millson}. We give a short conceptual proof.  Consider a trinode with edges $e_1, e_2, e_3$ weighted with $n_1, n_2, n_3$ which satisfy the Pieri rules (\ref{pieri}).  We can find a graph on $3$ vertices corresponding to the $3$ leaves of the trinode such that when we view the edges of the graph as passing through the trinode's edges, we recover $n_i$ as the number of paths passing through $e_i$.  The number of paths from $i$ to $j$ is $x_{ij} = \frac{1}{2}(n_i + n_j - n_k)$.  Notice that $x_{ij} + x_{ik} = n_i$.  This proves the lemma for the case $n = 3$.  Now take $\s \in S_\sigma$, and for each trivalent vertex $v \in V(\sigma)$, extract the paths associated to the edges connected to $v$.  For two vertices $v, v'$ connected by an edge $e$, this process yields the same number of paths in $e$, so we may glue these paths together any way we like.  The result is a graph on $n$ vertices (the leaves of $\sigma$).  Since this graph is a union of edges, $\s$ can be realized as a sum of the $\omega_{ij}$.  
\end{proof}

\subsection{Initial ideals from $\trop(I_{2, n})$}

Now we relate the associated graded algebras of the valuations $v_\r \in \mathcal{T}(n)$ to initial ideals $in_{\d(\r)}(I_{2, n})$ associated to points in the tropical variety $\trop(I_{2, n})$.  Let $\bx = \{x_{ij} \mid 1 \leq i < j \leq n\}.$

\begin{proposition}
For any $v_\r \in \mathcal{T}(n)$ the following hold:\\
\begin{enumerate}
\item the valuation $v_\r$ coincides with the weight quasi-valuation $\v_{\d(\r)}$ (see \ref{weightvaluation}),\\
\item the associated graded algebra $\gr_\r(\C[X])$ is isomorphic to $\C[\bx]/in_{\d(\r)}(I_{2, n})$,\\
\item if $\r \in C_\sigma \subset \mathcal{T}(n)$ satisfies $\r(e) \neq 0$, $\forall e \in E^{\circ}(\sigma)$, then $in_{\d(\r)}(I_{2, n})$
is the prime binomial ideal which vanishes on the generators $[\omega_{ij}]$ of the affine semigroup algebra $\C[S_\sigma]$. \\
\end{enumerate}

\end{proposition}

\begin{proof}
Part $(3)$ follows from $(2)$. Both $(1)$ and $(2)$ are a consequence of Theorem\ref{mainKhovanskii} and Proposition \ref{assgrad}.  
\end{proof}


\section{Maximal rank valuations and Newton-Okounkov cones of $X$}\label{toric}

In this section we use the divisor $D_\sigma \subset X_\sigma$ to construct maximal rank valuations on $\C[X]$, establishing $D_\sigma$ in the theory of \emph{Newton-Okounkov} bodies for the Grassmannian variety.  In particular, we show that $S_\sigma$ can be realized as the value semigroup of a valuation on $\C[X]$ which can be extracted from $D_\sigma$. 

\subsection{Maximal rank valuations on $\C[X]$}

There are many constructions of valuations on the Pl\"ucker algebra $\C[X]$ with representation-theoretic interpretations.  Alexeev and Brion \cite{Alexeev-Brion} give a construction in terms of Lusztig's dual canonical basis for any flag variety.  Kaveh \cite{Kaveh} then shows that the dual canonical basis construction can be recovered from a \emph{Parshin point} (see \ref{ppoint}) construction on a Bott-Samuelson resolution of the flag variety.  There are also many constructions of valuations coming from the theory of \emph{birational sequences}, which utilize the Lie algebra action  \cite{FFL}, \cite{Fang-Fourier-Littelmann},  and \cite{Fang-Fourier-Littelmann-birational}.  Finally, Cluster algebras (\cite{GHKK}, \cite{BFFHL}, \cite{RW}) provide another organizing tool for valuations.  The construction we give here is distinct from these approaches, and follows \cite{Manon-NOK} and can be derived from \cite{Kaveh-Manon-NOK}.

We pick a trivalent tree $\sigma$ and recall that branching basis $\B_\sigma \subset \C[X]$ constructed in Section \ref{assgradTn}.  Each member $b_\s \in \B_\sigma$ spans one of the spaces $W_\sigma(\s)$, and there is a bijection between the members of $\B_\sigma$ and the elements of the semigroup $S_\sigma$.   We select a total ordering $<$ on $E(\sigma)$; this induces a total ordering on $S_\sigma$ and the basis $\B_\sigma$ which we also denote by $<$.  In particular $\s < \s'$ if $-\s(e_i) < -\s'(e_i)$, where $e_i$ is the first edge (according to $<$) where $\s$ and $\s'$ disagree.  Now we define a function $\v_{\sigma, <}:\C[X] \setminus \{0\} \to \Z^{E(\sigma)}$ as follows:

\begin{equation}
\v_{\sigma, <}(\sum C_\s b_\s) = MIN\{\s \mid C_\s \neq 0\},\\
\end{equation}

\noindent
where $MIN$ is taken with respect to the ordering $<$ on $S_\sigma$. The following is essentially proved in \cite{Manon-NOK}.

\begin{proposition}\label{highrank}
The function $\v_{\sigma, <}$ is a discrete valuation on $\C[X]$ of rank $2n-3$ adapted to $\B_\sigma$ with value semigroup $S_\sigma$ and Newton-Okounkov cone $P_\sigma$. 
\end{proposition}

\begin{proof}
Let $F^{\sigma, <}_\s = \bigoplus_{\s < \s'}  W_\sigma(\s')$, then $F^{\sigma, <}_{\s} = \{ f \mid \v_{\sigma, <}(f) \geq \s\}$ by definition; this shows that $\v_{\sigma, <}$ is adapted to $\B_\sigma$.  Furthermore, for any $\s, \s'$ the product $F^{\sigma, <}_{\s} F^{\sigma, <}_{\s'}$ is a subspace of $F^{\sigma, <}_{\s + \s'}$ by Proposition \ref{multrule}.  This implies that $\v_{\sigma, <}$ is a quasi-valuation with value set equal to $S_\sigma$.  
To show that it is actually a valuation we observe that part $(4)$ of Proposition \ref{multrule} implies that $\v_{\sigma, <}(b_\s b_{\s'}) = \s + \s'$. This in turn implies that $\v_{\sigma, <}(fg) = \v_{\sigma, <}(f) + \v_{\sigma, <}(g)$, as $\v_{\sigma, <}$ will only see the values of the top components of $f$ and $g$ according to the ordering $<$. 
\end{proof}

It is also possible to show that there is a rank $|E(\sigma)|$ valuation $\v_{\sigma, <}: \C[X] \setminus \{0\} \to \Z^{E(\sigma)}$ for any non-trivalent $\sigma$.  In fact, one can apply \cite[Theorem 4]{Kaveh-Manon-NOK} to the integral generators of the extremal rays of $ev_n(C_\sigma) \subset \trop(I_{2, n})$ to recover any such valuation.  For this construction one needs to compute the values $v_{e_i}(p_{ij})$ for each edge $e_i \in E(\sigma)$ and each member of the Khovanskii basis of Pl\"ucker generators $p_{ij} \in \C[X]$.  This makes a $(2n-3) \times \binom{n}{2}$ matrix $M_{\sigma, <}$ which captures all the information of $\v_{\sigma, <}$. 



Proposition \ref{highrank} shows that the branching basis $\B_\sigma$ is adapted to the valuation $\v_{\sigma, <}$. Now, we describe a different basis which is also adapted to $\v_{\sigma, <}$.  Lemma \ref{finitesemigroup} shows that the Pl\"ucker generators $p_{ij}$ give a Khovanskii basis for $\v_{\sigma, <}$, so a basis of standard monomials in the $p_{ij}$ will be adapted to $\v_{\sigma, <}$ as well.  We say that $\sigma$ is a planar tree if the cyclic ordering on the leaves of $\sigma$ give an embedding of $\sigma$ into the plane.  In the proof of Lemma \ref{finitesemigroup} we can choose to decompose a weighting of $\sigma$ in a planar way, in particular it is always possible to construct the paths in a non-crossing way.  Furthermore, this reconstruction process shows that any two distinct planar arrangements give distinct weightings of $\sigma$.  Let $\B_+$ be the set of monomials $p^\alpha$ in the Pl\"ucker generators such that for any $i, j, k, \ell$ in cyclic order $\alpha_{ik}\alpha_{j\ell} = 0$; such monomials correspond to planar graphs on $[n]$ (\cite{HMSV}).  Our remarks imply the following proposition. 

\begin{proposition}\label{twobases}
The set $\B_+$ is an adapted basis of $\v_{\sigma, <}$ for any planar $\sigma$.  Furthermore, $\B_+$ and $\B_\sigma$ are related by upper-triangular transformations with respect the ordering on $S_\sigma$ induced by $<$.  



\end{proposition}

\subsection{A Parshin point construction of $\v_{\sigma, <}$}\label{ppoint}

In order to make a connection with the theory of Newton-Okounkov bodies, we present $\v_{\sigma, <}$ as a so-called Parshin point valuation (see \cite{Kaveh}, \cite{Kaveh-Khovanskii}, \cite{Lazarsfeld-Mustata}).  Roughly speaking, a Parshin point provides a higher rank generalization for the construction of a discrete valuation from a prime divisor on a normal variety.  Instead of taking degree along one height $1$ prime, one takes successive degrees along a flag of subvarieties.

\begin{definition}[Parshin point valuation]
 Let $p \in Y$ be a point in a variety of dimension $dim(Y) = n$, and $V_1 \supset \ldots \supset V_n = \{p\}$ be a flag of irreducible subvarieties.  We further assume that $V_i$ is locally cut out of $V_{i-1}$ at $p$ by $t_i$. This information defines a Parshin point, which we denote by $\mathcal{F}$.  For $f \in \C(Y)$ we define a valuation $\v_\mathcal{F}$ as follows.  Let $s_1 = ord_{t_1}(f)$, $f_1 = t_1^{-s_1}f \mid_{V_1}$ and then continue this way to get $s_i = ord_{t_i}(f_{i-1})$, $f_i = t_i^{-s_i}f_{i-1} \mid_{V_i}$.  We set $\v_\mathcal{F}(f) = (s_1, \ldots, s_n)$. 
\end{definition}

Let $\sigma$ be a trivalent tree, and let $<$ be a total ordering on $E(\sigma)$.  We use $<$ to define a flag of subvarieties on $X_\sigma$.  Using the total ordering $<$, we can label the edges $E(\sigma)$: $e_1, \ldots, e_{2n-3}$.  This defines a flag $D_{e_1} \supset D_{e_1, e_2} \supset \ldots \supset D_{E(\sigma)}$, where $D_S$ is the subvariety defined in Section \ref{divisorstructuresection}.  In particular, $D_{E(\sigma)}$ is the point in $X_\sigma$ defined by the maximal ideal $I_{E(\sigma)}$.

\begin{proposition}\label{flagparshinpoint}
The Parshin point $D_{e_1} \supset D_{e_1, e_2} \supset \ldots \supset D_{E(\sigma)}$ defines a valuation $\w_{\sigma, <}: \C[X]\setminus\{0\} \to \Z^{E(\sigma)}$ which coincides with $\v_{\sigma, <}$. 
\end{proposition}

\begin{proof}
We start by considering the ideal $J_{E(\sigma)} \subset \C[\overline{M}(\sigma)]$:

\begin{equation}
J_{E(\sigma)} = \bigoplus_{n\ge 0}\bigoplus_{\{\s\in\Z_{\ge 0}^{E(\sigma)} |\forall e\in E(\sigma),\s(e)\leq n, \ \exists e', \s(e')<n\}} \big[ \bigotimes_{e \in E^{\circ}(\sigma)} V(\s(e)) \otimes V(\s(e))\big] \otimes \big[\bigotimes_{\ell \in L(\sigma)} V(\s(\ell)) \big] t^n.\\
\end{equation}

\noindent
Notice that $J_{E(\sigma)} \cap \C[X_\sigma] = I_{E(\sigma)}$, $J_{E(\sigma)}$ is $G(\sigma)$-fixed, and $\O_{J_{E(\sigma)}}^{G(\sigma)} = \O_{I_{E(\sigma)}}$, where these are the local rings for the corresponding points on $\overline{M}(\sigma)$ and $X_\sigma$, respectively.   We work with the space $\overline{M}(\sigma)$ because it has the advantage of being smooth. Let  $t_e$ be the local equation for the prime divisor $\overline{M}(\sigma, e)$ (this element can be taken to coincide with $\bar{t}_e$ from Proposition \ref{valscoincide}).  Note that $t_e$ is $G(\sigma)$-fixed, so $t_e \in \O_{I_{E(\sigma)}}$.  Furthermore $\overline{M}(\sigma)$ is a product over $e \in E(\sigma)$, so we must have $ord_{t_e}(t_{e'})= 0$ when $e \neq e'$.  The subvarieties $D_{e_1, \ldots, e_k}$ with their local equations $t_{e_k}$ define a Parshin point of $X_\sigma$; in particular $t_{e_k}$ locally cuts out $D_{e_1, \ldots, e_k}$ in $D_{e_1, \ldots, e_{k-1}}$, because this is the case for the corresponding ideals in $\O_{J_{E(\sigma)}}$. 

We construct $D_S$ as a GIT quotient of the space $\overline{M}(\sigma, S)$ in \ref{DSgeom}.  We let $\overline{M}(\sigma, S)^o \subset \overline{M}(\sigma, S)$ be the subvariety obtained by the same product construction, only replacing $\overline{\SL}_2$ with $\SL_2$ (resp. $\P^2$ with $\A^2$ where appropriate) whenever $e \ (resp. \ \ell) \notin S$.    Now we choose $b_\s \in \C[X]$.  Proposition \ref{valscoincide} shows that $ord_{t_{e_1}} = -\s(e_1)$.  Regarding $t_{e_1}^{\s(e_1)}b_\s$ as a function on $\overline{M}(\sigma, e_1)^o$, we use the same argument in \ref{valscoincide} to show that $ord_{t_{e_2}}(t_{e_1}^{\s(e_1)}b_\s) = 0 - \s(e_2)$, where $t_{e_2}$ locally cuts out $\overline{M}(\sigma, e_1, e_2)$ along the boundary of $\overline{M}(\sigma, e_1)^o$.  Continuing this way, we obtain the valuation $\w_{\sigma, <}$, which has the property that $\w_{\sigma, <}(b_\s) = \v_{\sigma, <}(b_\s)$ for any $b_s \in \B_\sigma$.  Since both valuations are maximal rank, and since they take the same distinct values on the basis $\B_\sigma$, they must coincide (see Proposition \ref{adaptedbasisproperties}). 
\end{proof}

\section{Example}
Let's consider the simplest case: $n=4$.
There are three trivalent trees with 4 ordered leaves (Figures 4, 5, 6). 
The Tropical Grassmannian $\trop(I_{2, 4})\subset \R^6$ is a fan with three $5$-dimensional cones $\overline{B}_{\sigma_1}, \overline{B}_{\sigma_2}, \overline{B}_{\sigma_3}$  glued along a $4$-dimensional lineality space which can be constructed as the image of the map \[\psi:\R^4\rightarrow\R^6, (x_1, x_2, x_3, x_4)\mapsto (x_1+x_2, x_1+x_3, x_1+x_4, x_2+x_3, x_2+x_4, x_3+x_4): \]

\[\overline{B}_{\sigma_1}: w_{13}+w_{24}=w_{14}+w_{23}\le w_{12}+w_{34},\quad\mathrm{image}(\psi)+\R_{\ge 0} e_{13}+e_{14}+e_{23}+e_{24}\]
\[\overline{B}_{\sigma_2}: w_{12}+w_{34}=w_{14}+w_{23}\le w_{13}+w_{24},  \quad\mathrm{image}(\psi)+\R_{\ge 0}e_{12}+e_{14}+e_{23}+e_{34}\]
\[\overline{B}_{\sigma_3}: w_{12}+w_{34}=w_{13}+w_{24}\le w_{14}+w_{23},  \quad\mathrm{image}(\psi)+\R_{\ge 0} e_{12}+e_{13}+e_{24}+e_{34},\] where 
$w_{ij}$ are the coordinates of $\R^6=\R^{4\choose 2}$ and  $e_{ij}$ are the standard basis of $\R^6=\R^{4\choose 2}$.  Note that $\mathrm{image}(\psi)$ is spanned by 
\smallskip
$\psi(e_1)=(1, 1, 1, 0, 0, 0),$ 
$\psi(e_2)=(1, 0, 0, 1, 1, 0),$
$\psi(e_3)=(0, 1, 0, 1, 0, 1),$
$\psi(e_4)=(0, 0, 1, 0, 1, 1)$.

\smallskip
\begin{figure}[!htb]\centering
\begin{minipage}{0.32\textwidth}
\begin{tikzpicture}[
every edge/.style = {draw=black,very thick, ->},
 vrtx/.style args = {#1/#2}{%
      circle, draw, thick, fill=white,
      minimum size=5/2mm, label=#1:#2}]
\node(A) [vrtx=left/] at (0, 0) {$v_2$};
\node(B) [vrtx=left/] at (0,4/2) {$v_1$};
\node(C) [vrtx=left/] at (-3.46/2,-2/2) {$\ell_4$};
\node(D) [vrtx=left/] at (3.46/2, -2/2) {$\ell_3$};
\node(E) [vrtx=left/] at (-3.46/2,6/2){$\ell_1$};
\node(F) [vrtx=left/] at (3.46/2,6/2){$\ell_2$};

\path   (A) edge (B)
           (A) edge (C)
	(A) edge (D)
	(B) edge (E)
	(B) edge (F);
\end{tikzpicture}
\caption{ $\sigma_1=(\{1,2\},\{3,4\})$}
\end{minipage}
\begin{minipage}{0.32\textwidth}
\begin{tikzpicture}[
every edge/.style = {draw=black,very thick, ->},
 vrtx/.style args = {#1/#2}{%
      circle, draw, thick, fill=white,
      minimum size=5/2mm, label=#1:#2}]
\node(A) [vrtx=left/] at (0, 0) {$v_2$};
\node(B) [vrtx=left/] at (0,4/2) {$v_1$};
\node(C) [vrtx=left/] at (-3.46/2,-2/2) {$\ell_2$};
\node(D) [vrtx=left/] at (3.46/2, -2/2) {$\ell_4$};
\node(E) [vrtx=left/] at (-3.46/2,6/2){$\ell_1$};
\node(F) [vrtx=left/] at (3.46/2,6/2){$\ell_3$};

\path   (A) edge (B)
           (A) edge (C)
	(A) edge (D)
	(B) edge (E)
	(B) edge (F);
\end{tikzpicture}
\caption{ $\sigma_2=(\{1,3\},\{2,4\})$}
\end{minipage}
\begin{minipage}{0.32\textwidth}
\begin{tikzpicture}[
every edge/.style = {draw=black,very thick, ->},
 vrtx/.style args = {#1/#2}{%
      circle, draw, thick, fill=white,
      minimum size=5/2mm, label=#1:#2}]
\node(A) [vrtx=left/] at (0, 0) {$v_2$};
\node(B) [vrtx=left/] at (0,4/2) {$v_1$};
\node(C) [vrtx=left/] at (-3.46/2,-2/2) {$\ell_2$};
\node(D) [vrtx=left/] at (3.46/2, -2/2) {$\ell_3$};
\node(E) [vrtx=left/] at (-3.46/2,6/2){$\ell_1$};
\node(F) [vrtx=left/] at (3.46/2,6/2){$\ell_4$};

\path   (A) edge (B)
           (A) edge (C)
	(A) edge (D)
	(B) edge (E)
	(B) edge (F);
\end{tikzpicture}
\caption{$\sigma_3=(\{1,4\},\{2,3\})$}
\end{minipage}
\end{figure}

Now, we see how to recover the tropical geometry and Newton-Okounkov information from the compactification $X\subset X_{\sigma_1} \supset D_{\sigma_1}$.  From each component $D_e \subset D_{\sigma_1},  (e\in E(\sigma_1))$, we get a vector in  $\trop(I_{2, 4})\subset \R^6$ using Proposition \ref{valscoincide} and Proposition \ref{pluckerevaluate}: 

\[\begin{array}{ccl}
D_{\ell_1}&\mapsto&(ord_{D_{\ell_1}}(p_{12}), ord_{D_{\ell_1}}(p_{13}), ord_{D_{\ell_1}}(p_{14}), ord_{D_{\ell_1}}(p_{23}), ord_{D_{\ell_1}}(p_{24}), ord_{D_{\ell_1}}(p_{34}))\\
              && =  (v_{{\ell_1}}(p_{12}), v_{{\ell_1}}(p_{13}), v_{{\ell_1}}(p_{14}), v_{{\ell_1}}(p_{23}), v_{{\ell_1}}(p_{24}), v_{{\ell_1}}(p_{34}))
= (1, 1, 1, 0, 0, 0)=\psi(e_1)\\

D_{\ell_2}&\mapsto&(ord_{D_{\ell_2}}(p_{12}), ord_{D_{\ell_2}}(p_{13}), ord_{D_{\ell_2}}(p_{14}), ord_{D_{\ell_2}}(p_{23}), ord_{D_{\ell_2}}(p_{24}), ord_{D_{\ell_2}}(p_{34}))\\
              && =  (v_{{\ell_2}}(p_{12}), v_{{\ell_2}}(p_{13}), v_{{\ell_2}}(p_{14}), v_{{\ell_2}}(p_{23}), v_{{\ell_2}}(p_{24}), v_{{\ell_2}}(p_{34}))
= (1, 0, 0, 1, 1, 0)=\psi(e_2)\\

D_{\ell_3}&\mapsto&(ord_{D_{\ell_3}}(p_{12}), ord_{D_{\ell_3}}(p_{13}), ord_{D_{\ell_3}}(p_{14}), ord_{D_{\ell_3}}(p_{23}), ord_{D_{\ell_3}}(p_{24}), ord_{D_{\ell_3}}(p_{34}))\\
              && =  (v_{{\ell_3}}(p_{12}), v_{{\ell_3}}(p_{13}), v_{{\ell_3}}(p_{14}), v_{{\ell_3}}(p_{23}), v_{{\ell_3}}(p_{24}), v_{{\ell_3}}(p_{34}))
= (0, 1, 0, 1, 0, 1)=\psi(e_3)\\

D_{\ell_4}&\mapsto&(ord_{D_{\ell_4}}(p_{12}), ord_{D_{\ell_4}}(p_{13}), ord_{D_{\ell_4}}(p_{14}), ord_{D_{\ell_4}}(p_{23}), ord_{D_{\ell_4}}(p_{24}), ord_{D_{\ell_4}}(p_{34}))\\
              && =  (v_{{\ell_4}}(p_{12}), v_{{\ell_4}}(p_{13}), v_{{\ell_4}}(p_{14}), v_{{\ell_4}}(p_{23}), v_{{\ell_4}}(p_{24}), v_{{\ell_4}}(p_{34}))
= (0, 0, 1, 0, 1, 1)=\psi(e_4)\\
D_{e^{\circ}}&\mapsto&(ord_{D_{e^{\circ}}}(p_{12}), ord_{D_{e^{\circ}}}(p_{13}), ord_{D_{e^{\circ}}}(p_{14}), ord_{D_{e^{\circ}}}(p_{23}), ord_{D_{e^{\circ}}}(p_{24}), ord_{D_{e^{\circ}}}(p_{34}))\\
              && =  (v_{{e^{\circ}}}(p_{12}), v_{{e^{\circ}}}(p_{13}), v_{{e^{\circ}}}(p_{14}), v_{{e^{\circ}}}(p_{23}), v_{{e^{\circ}}}(p_{24}), v_{{e^{\circ}}}(p_{34}))
= (0,1,1,1,1,0)=e_{13}+e_{14}+e_{23}+e_{24},\\
\end{array}\]

where $e^{\circ}$ is the non-leaf edge.  

We describe them as row vectors of the following matrix, which span the maximal cone $\overline{B}_{\sigma_1}$ of $\trop(I_{2, 4})$.

\[ \left(\begin{array}{c}
 ord_{D_{\ell_1}}\\
 ord_{D_{\ell_2}}\\
 ord_{D_{\ell_3}}\\
 ord_{D_{\ell_4}}\\
 ord_{D_{e^{\circ}}}
\end{array}\right)
=
\begin{blockarray}{cccccc}
p_{12}&p_{13}&p_{14}&p_{23}&p_{24}&p_{34}\\
\begin{block}{(cccccc)}
1&1&1&0&0&0\\
1&0&0&1&1&0\\
0&1&0&1&0&1\\
0&0&1&0&1&1\\
0&1&1&1&1&0\\
\end{block}
\end{blockarray}\]

For Newton-Okounkov information, we fix a total order $<$ on $E(\tree)$, for example,  $\ell_1>\ell_2>\ell_3>\ell_4>e^{\circ}$, the order we used for the tropical geometry above.  It corresponds to the flag of subvarieties (Parshin point), $D_{\ell_1}\supset D_{\ell_1, \ell_2}\supset\cdots\supset D_{E(\sigma)}$ and the valuation $\v_{\sigma, <}$ (see Proposition \ref{flagparshinpoint}). Now we compute the values of Pl\"ucker generators under this valuation:

$ \v_{\tree,<}(p_{12})=(1, 1, 0, 0, 0) $.

$\v_{\tree,<}(p_{13})=(1, 0, 1, 0, 1) $.

$\v_{\tree,<}(p_{14})=(1, 0, 0, 1, 1) $.

$\v_{\tree,<}(p_{23})=(0, 1, 1, 0, 1) $.

$\v_{\tree,<}(p_{24})=(0, 1, 0, 1, 1) $

$\v_{\tree,<}(p_{34})=(0, 0, 1, 1, 0) $.

These coincide with the column vectors of the matrix above, which generate the semigroup $S_{\sigma}$. Thus from the perspective of the compactification $X_{\tree}$, we have a unified understanding of the tropical geometry and Newton-Okounkov theory for the affine cone $X$ of the Grassmannian $Gr_2(\C^4)$.

Also, these generators of $S_{\sigma}$ correspond to the parametrization of the toric variety defined by  $\C[S_{\sigma}]$:

\[(\C^*)^5\rightarrow \C^6, (t_1, t_2, t_3, t_4, t_5)\mapsto (t_1t_2, t_1t_3t_5, t_1t_4t_5, t_2t_3t_5, t_2t_4t_5, t_3t_4)\]

The ideal presenting $\C[S_{\sigma}]$ is equal to the kernel of the homomorphism, 
\[\C[p_{12}, p_{13}, p_{14},p_{23}, p_{24}, p_{34}]\rightarrow \C[y_{e_1}, y_{e_2}, y_{e_3}, y_{e_4}, y_{e_5}], p_{12}\mapsto y_{e_1}y_{e_2}, p_{13}\mapsto y_{e_1}y_{e_3}y_{e_5},..., p_{34}\mapsto y_{e_3}y_{e_4},\] which is the 
 the principal ideal generated by $p_{13}p_{24}-p_{14}p_{23}$. Now, this ideal is equal to the initial ideal $in_{\sigma}(I_{2, 4})$ associated to the cone $B_{\sigma}$ in  $\trop(I_{2, 4})$.

\bibliographystyle{alpha}
\bibliography{Arxiv_Compactifications_and_the_Grassmannian}

\end{document}